\documentclass[english]{amsart}
\usepackage{amsfonts,amssymb,amsmath,amsgen,amsthm}
\usepackage{hyperref,color}
\usepackage{pdfsync}
\usepackage{mathtools}

\makeatletter
\theoremstyle{plain}
\newtheorem{thm}{\protect\theoremname}
\theoremstyle{definition}

\theoremstyle{remark}
\newtheorem{rem}[thm]{\protect\remarkname}
\theoremstyle{plain}

\theoremstyle{plain}
\newtheorem{prop}[thm]{\protect\propositionname}

\theoremstyle{plain}
\newtheorem{assumption}{Assumption}

\makeatother

\usepackage{babel}
\providecommand{\definitionname}{Definition}
\providecommand{\lemmaname}{Lemma}
\providecommand{\propositionname}{Proposition}
\providecommand{\remarkname}{Remark}
\providecommand{\theoremname}{Theorem}

\def\R{{\mathbf R}}
\def\T{{\mathbf T}}
\def\N{{\mathbf N}}
\def\Z{{\mathbf Z}}
\def\d{{\partial}}

\DeclareMathOperator{\diver}{div}

\numberwithin{equation}{section}

\date\today

\title{Initial layer analysis of relaxation-time limit of the collisional QHD}

\author[P. Antonelli]{Paolo Antonelli}
\address{Gran Sasso Science Institute, viale Francesco Crispi, 7, 67100 L'Aquila, Italy}
\email{paolo.antonelli@gssi.it}

\author[P. Marcati]{Pierangelo Marcati}
\address{Gran Sasso Science Institute, viale Francesco Crispi, 7, 67100 L'Aquila, Italy}
\email{pierangelo.marcati@gssi.it}

\author[H. Zheng]{Hao Zheng}
\address{Chinese Academy of Science, Zhongguancun Est. Rd., Haidian District, Beijing, China, 100190}
\email{zhenghao@amss.ac.cn}


\begin{document}

\maketitle

\begin{abstract}
We study the structure of the initial layer arising in the relaxation-time limit of the collisional quantum hydrodynamic (QHD) system. When the initial data are not well prepared, a fast transient regime appears near the initial time, which prevents the uniform-in-time convergence of the momentum density to its limiting value.

Using the method of matched asymptotic expansions, we derive a systematic asymptotic description of the solution with respect to the relaxation-time parameter $\tau$. In particular, we identify the fast time scale $t/\tau^{2}$ governing the initial layer and explicitly construct the corresponding inner expansion for the momentum density together with the outer expansion describing the slow dynamics. The leading-order outer dynamics are shown to coincide with the quantum drift–diffusion equation.

The asymptotic expansion is rigorously justified by establishing uniform in $\tau$ energy estimates for the remainder terms under suitable regularity assumptions on the solutions. As a consequence, we prove the strong convergence of the momentum density in $L^\infty$ in time after subtracting the leading initial-layer correction. The analysis further shows that the convergence rate of order $\tau$ is optimal for general initial data and explains the improved rate in the well-prepared case.
\end{abstract}

\section{Introduction and main results}

The main purpose of this paper is to analyze the asymptotic structure of the time-relaxation limit of the following quantum hydrodynamic (QHD) system
\begin{equation}\label{eq:QHD}
\begin{cases}
\d_{t'}\rho+\diver J=0\\
\d_{t'}J+\diver\left(\frac{J\otimes J}{\rho}\right)+\nabla p(\rho)+\rho\nabla V=\frac{1}{2}\rho\nabla\left(\frac{\triangle\sqrt\rho}{\sqrt\rho}\right)-\frac{1}{\tau}J\\
-\triangle^2V=\rho-\mathcal{C}(x),\quad (\rho,J)(0,x)=(\rho_0,J_0)(x),
\end{cases}
\end{equation}
for $(t',x)\in[0,\infty)\times {\T^d}$, where ${\T^d}=\R^d/\Z^d$ is a $d-$dimensional torus, with $1\le d\leq 3$. The unknowns $\rho$ and $J$ in \eqref{eq:QHD} represent the mass (charge) and the momentum (current) densities of the quantum fluid, respectively, and $p(\rho)$ is the isentropic pressure term. The parameter $\tau>0$ denotes the scaled momentum relaxation time. Among the physical and engineering applications of system \eqref{eq:QHD}, it is widely used to model carrier transport in semiconductor devices \cite{J} at nanoscales, when quantum effect must be taking into account \cite{AI, G}. 
The term $-\frac{1}{\tau}J$ on the right-hand side of the momentum density equation introduces a dissipative effect in the system, that phenomenologically describes electron collisions in the semiconductor device \cite{AT, BW}.
The function $V$ denotes a self-consistent electric potential, satisfying the Poisson equation, and $\mathcal{C}(x)$ denotes the density of background positively charged ions. 
The quantum stress tensor satisfies the following identities
\begin{equation}\label{eq:bohm}
\frac12\rho\nabla\left(\frac{\triangle\sqrt{\rho}}{\sqrt{\rho}}\right)=\frac14\diver(\rho\nabla^2\log\rho)=\frac14\nabla\triangle\rho-\diver(\nabla\sqrt{\rho}\otimes\nabla\sqrt\rho).
\end{equation}
The left hand side is commonly interpreted as a quantum pressure expressed in terms of a quantum enthalpy, while the right hand side provides the decoupling into a linear dispersive tensor and a quadratic term taking into account the Fisher information \cite{HC}.

We are interested in the time-relaxation limit, namely we analyze the asymptotic dynamics in the limit when $\tau\to0$. For this purpose, we introduce the scaling
\begin{equation}\label{eq:rs_intro}
t=\tau t',\quad (\rho_\tau,J_\tau)(t,x)=\left(\rho,\frac{1}{\tau}J\right)\left(\frac{t}{\tau},x\right),
\end{equation}
then the system \eqref{eq:QHD} can be reformulated as
\begin{equation}\label{eq:QHD_rs_intro}
\left\{\begin{aligned}
&\d_{t}\rho_\tau+\diver J_\tau=0\\
&\tau^2\d_{t} J_\tau+\tau^2\diver\left(\frac{J_\tau\otimes J_\tau}{\rho_\tau}\right)+\nabla p(\rho_\tau)+\rho_\tau\nabla V_\tau=\frac{1}{2}\rho_\tau\nabla\left(\frac{\triangle\sqrt\rho_\tau}{\sqrt\rho_\tau}\right)-J_\tau\\
&-\triangle V_\tau=\rho_\tau-\mathcal{C}(x),\quad (\rho_\tau,J_\tau)(0,x)=(\rho_0,\tau^{-1}J_0)(x).
\end{aligned}\right.
\end{equation}
In this paper, we restrict ourselves to a finite rescaled time interval $t\in [0,T_0]$.
As $\tau\to 0$, system \eqref{eq:QHD_rs_intro} formally converges to the following Quantum Drift-diffusion (QDD) equation, 
\begin{equation}\label{eq:qdde}
\begin{cases}
\d_{t}\bar\rho+\diver\left[\frac{1}{2}\bar\rho\nabla\left(\frac{\triangle\sqrt{\bar\rho}}{\sqrt{\bar\rho}}\right)-\nabla p(\bar\rho)-\bar\rho\nabla V(\bar\rho)\right]=0\\
-\triangle V(\bar\rho)=\bar\rho-\mathcal{C}(x),\quad \bar\rho(0,x)=\rho_0(x).
\end{cases}
\end{equation}

The relaxation-time limit for the QHD system \eqref{eq:QHD} was first considered by \cite{JLM} for small and smooth perturbations around a constant state, and later extended in \cite{LZZ} to the bipolar case, namely the system coupling of electron and ion dynamics. In \cite{LT} and the references therein, the relaxation-time limit was analyzed for finite energy weak solutions by using relative energy method, however they need to assume the initial data to be well-prepared (see the next paragraph) and the solutions of the limiting equations \eqref{eq:qdde} to be sufficiently smooth. In the literature of the classical hydrodynamic equations, the study of relaxation-time limits is also widely performed, see for example \cite{HL,HP,JP, MN}. We also refer to \cite{DM} for a general theory of similar singular limits. In the presence of a viscous stress tensor, that in the case of \eqref{eq:QHD} consists in considering the so-called quantum Navier-Stokes equatio, the time-relaxation limit may be analyzied also in the framework of finite energy weak solution, through the derivation of suitable a priori estimates \cite{ACLS, Blub}.

In the case of the initial data that is not well-prepared, an important question that should be investigated is the phenomenon of initial layer. The initial data for the rescaled QHD system \eqref{eq:QHD_rs_intro} consists of two components, namely the initial density $\rho_0$ and the initial momentum density $J_{\tau,0}$. On the other hand, to solve the QDD equation \eqref{eq:qdde}, we only need to prescribe the density $\rho_0$ at time $t=0$, while the initial momentum density does not necessarily satisfies the limiting momentum density is required to satisfy the constitutive relation
\begin{equation}\label{eq:barJ}
\bar{J}=\frac{1}{2}\bar\rho\nabla\left(\frac{\triangle\sqrt{\bar\rho}}{\sqrt{\bar\rho}}\right)-\nabla p(\bar\rho)-\bar\rho\nabla V(\bar\rho-\mathcal{C}(x)).
\end{equation} 
Due to this mismatch of initial data, a thin layer appears near the initial time, where the momentum density in general largely deviates from the limit $\bar J$. 
The initial data for system \eqref{eq:QHD_rs_intro} is called \textit{well-prepared} if 
\begin{equation}\label{eq:cond_wellpre}
J_{\tau, 0}=\tau^{-1}J_0\to\bar J_0\quad\textrm{as }\tau\to 0,
\end{equation}
 where $\bar J_0$ is defined by letting $\bar\rho=\rho_0$ through \eqref{eq:barJ}. In this case, no initial layer is expected. 

Recently, in the authors' work \cite{AMZ3}, the relaxation-time limit was rigorously proven for general ill-prepared initial data on a 1-D torus $\T$. This result is compatible with  the phenomenon of initial layer. Moreover, in \cite{AMZ3} an explicit rate of convergence with respect to the relaxation-time parameter $\tau\to 0^+$ was obtained, namely
\[
\|\rho_\tau-\bar\rho\|_{L^\infty_tL^2_x}+\|\sqrt\rho_\tau (\d_x^2\log\sqrt\rho_\tau-\d_x^2\log\sqrt{\bar\rho})\|_{L^2_{t,x}}\leq C\tau,
\]
where $C>0$ is a constant independent of $\tau$. We also mention that it is shown in \cite{LT} that for well-prepared initial data, the convergence rate can be improved to $\tau^2$. On the other hand, the convergence of the momentum density $J_\tau$ to the limit $\bar J$ was only given in a weak $L^2_{t,x}$ sense,
\[
J_\tau \rightharpoonup \bar J \quad\textrm{in }L^2_{t,x}.
\]
To obtain a strong convergence result of the momentum density $J_\tau$, it is necessary to unveil the structure of the initial layer. By introducing a correction term $J_\mathcal{I}$ which characterizes the dynamics of the initial layer, it will allow us to prove the strong convergence of $J_\tau$ to $J_\mathcal{I}+\bar J$.

\subsection{Matched asymptotic expansions and main results}

The main purpose of this paper is to conduct a deeper analysis of the asymptotic structure of the time-relaxation limit and the initial layer. We examine the asymptotic expansion of $(\rho_\tau, J_\tau)$ with respect to the relaxation time $\tau$, both within the initial layer (inner expansion) and in regions away from the initial time (outer expansion). The inner expansion reveals the structure of the sharp transition in the momentum density $J_\tau$ near the initial time, thereby enabling an explicit characterization of the norm convergence of $J_\tau$ in a suitable Lebesgue space. Meanwhile, the outer expansion offers a more detailed description of the convergence of $(\rho_\tau, J_\tau)$ to $(\bar\rho, \bar J)$ at each order of $\tau$. As a result of this analysis, we demonstrate that for general initial data, the convergence rate of order $\tau$ established in \cite{AMZ3} is indeed optimal. Furthermore, by deriving energy estimates for the remainder terms that are uniform in $\tau$, we rigorously justify both the inner and outer asymptotic expansions. An additional outcome of our analysis is the illustration of how the convergence rate improves in the case of well-prepared initial data. Our techniques are related to the approach developed by Lattanzio and Yong for general hyperbolic balance laws \cite{LY}.

Let us consider the formal asymptotic approximation of the rescaled solutions $(\rho_\tau,J_\tau)$ of the QHD system \eqref{eq:QHD_rs_intro}, given in the form
\begin{equation}\label{eq:exptau}
\begin{aligned}
\rho_\tau(t,x)=\rho_\mathcal{O}(t,x)+\rho_\mathcal{I}(s,x)-P_\rho(s,x)+\tau^2 r(s,t,x)\\
J_\tau(t,x)= J_\mathcal{O}(t,x)+J_\mathcal{I}(s,x)-P_J(s,x)+\tau^2\mathcal{R}(s,t,x),
\end{aligned}
\end{equation}
where the fast time in general in given in the form $s=\tau^{-n} t$, $n\in\N$ the time scale of the initial layer. In what follows we will show that the initial layer phenomenon is captured by $n=2$. In general, inner expansion focuses on certain narrow regions (like boundary layers or initial layers) where the solution changes at very fast scale, while the outer expansion provides an approximation far away from the initial time, where the dynamics is governed only by the slow scale. To perform the analysis of the initial layer, it is sufficient to consider the expansion to the first order in $\tau$. Inner functions can be defined as
\begin{equation}\label{eq:inner}
\rho_\mathcal{I}(s,x)=\sum_{k=0}^1\tau^{k}\rho_{\mathcal{I},k}(s,x),\quad J_\mathcal{I}(s,x)=\sum_{k=-1}^1\tau^{k}J_{\mathcal{I},k}(s,x),
\end{equation}
while the outer functions are given by
\begin{equation}\label{eq:outer}
\rho_\mathcal{O}(s,x)=\sum_{k=0}^1\tau^{k}\rho_{\mathcal{O},k}(s,x),\quad J_\mathcal{O}(s,x)=\sum_{k=0}^1\tau^{k}J_{\mathcal{O},k}(s,x).
\end{equation}
The leading order of the inner momentum density is chosen to be $\tau^{-1}$ in order to match the scaling \eqref{eq:rs_intro} and general rescaled initial momentum density
\[
J_\tau(0)=\tau^{-1}J_0.
\]


In the asymptotic expansion \eqref{eq:exptau}, the functions $P_\rho$, $P_J$ are matching functions which will be used to connect the initial layer functions and the outer functions in a consistent way. In general cases, the matching function $P_f=\sum_k P_{f,k}$, corresponding either $f=\rho,J$, is defined as
\begin{equation}\label{eq:match_intro}
P_{f,k}(s,x)=\sum_{j=0}^{\lfloor k/n\rfloor}\frac{s^j}{j!}\left.\d_t^j f_{\mathcal{O},k-nj}\right|_{t=0},
\end{equation}
where $\lfloor k/n\rfloor$ is the integer part of $k/n$, and the following matching condition holds
\begin{equation}\label{eq:cond_match_intro}
\lim_{s\to\infty}(\rho_{\mathcal{I},k},J_{\mathcal{I},k})(s)=\lim_{s\to\infty}(P_{\rho,k},P_{J,k})(s),\quad k=0,1.
\end{equation}
In our case, \eqref{eq:match_intro} simplifies to
\begin{equation}
P_{f,k}(s,x)=\left. f_{\mathcal{O},k}\right|_{t=0},\quad k=0,1.
\end{equation}
for either $f=\rho,J$.

In the subsequent sections of this paper, we derive the explicit forms of the initial layer and the outer region expansion $ (\rho_{\mathcal{I},k}, J_{\mathcal{I},k}) $ and $ (\rho_{\mathcal{O},k}, J_{\mathcal{O},k}) $, respectively, by applying the method of matched asymptotic expansions \cite{CK} to the QHD system \eqref{eq:QHD_rs_intro}. The dynamics of the initial layer are then governed by the ODE systems given in \eqref{eq:innereq_rho} and \eqref{eq:innereq_J} for $ (\rho_{\mathcal{I},k}, J_{\mathcal{I},k}) $. The matching conditions \eqref{eq:cond_match_intro} ensure that the differences $ (\rho_{\mathcal{O},k} - P_{\rho,k}, J_{\mathcal{O},k} - P_{J,k}) $ converge to zero as the fast-time variable $ s = \tau^{-2}t $ tends to infinity. In particular, the leading-order behavior of the initial layer in the momentum density is given by
\[
\tau^{-1}J_{\mathcal{I},-1} = \tau^{-1} e^{-\frac{t}{\tau^2}} J_0.
\]
Furthermore, as shown in Section \ref{sect:outer}, the zeroth-order outer term $ \rho_{\mathcal{O},0} $ is exactly the solution $\bar\rho$ to the Cauchy problem of the quantum drift-diffusion (QDD) equation \eqref{eq:qdde}, while $ \rho_{\mathcal{O},1} $ satisfies a semilinear fourth-order parabolic equation defined in \eqref{eq:1equ} later, which is the linearized QDD dynamics around $\bar\rho$. 

Finally, we employ energy methods to establish the uniform boundedness (with respect to $ \tau $) of remainder terms $ (r, \mathcal{R}) $ under suitable regularity assumptions. Let us assume the densities are strictly positive,
\begin{equation}\label{eq:pos}
\inf_{t,x}\,\{\rho_0,\rho_\tau,\bar\rho\}\geq\delta >0
\end{equation}
and the initial velocity $v_0=\frac{J_0}{\rho_0}$ is irrotational,
\begin{equation}\label{eq:irr}
(\nabla v_0)^a=\frac12\left(\nabla v_0-\nabla v_0^T\right)=0.
\end{equation}
Moreover, we also assume that there exists a $M>0$ such that 
\begin{equation}\label{eq:reg}
\max\left\{\|\rho_0\|_{H^{9}_x},\|J_0\|_{H^{8}_x},\|\rho_\tau\|_{L^\infty_x H^{9}_x},\|J_\tau\|_{L^\infty_x H^{8}_x},\|\bar\rho\|_{L^\infty_x H^{9}_x},\|\rho_{\mathcal{O},1}\|_{L^\infty_x H^{7}_x}\right\}\leq M.
\end{equation}
We emphasize that in this paper we will not provide new results concerning the existence of solution satisfying assumption \eqref{eq:pos} and \eqref{eq:reg}, but we will confine ourselves to establish the analysis of the initial layer within such framework. For more details of these assumptions, see Assumption \ref{asmp:RA} in Section \ref{sect:en_remainder}. We also address the interested reader to \cite{AMZ1, AM_K}, for a more comprehensive discussion regarding the existence of QHD-type systems.

\begin{thm}\label{thm:main}
Let us assume \eqref{eq:pos}, \eqref{eq:irr} and \eqref{eq:reg} hold on rescaled time interval $[0,T_0]$. Then there exist $\tau^*>0$ and $C>0$, dependent on $(T_0,\delta,M)$, such that for $0<\tau<\tau^*$ the remainders $(r,\mathcal{R})$ are bounded by
\begin{equation}\label{eq:main_r}
\|r\|_{L^\infty_tH^1_x}+\|\tau\mathcal{R}\|_{L^\infty_tL^2_x}+\|\mathcal{R}\|_{L^2_{t,x}}\leq C
\end{equation}
on the time interval $[0,T_0]$.

As a consequence, we have
\begin{equation}\label{eq:main_J}
\|J_\tau-\tau^{-1}e^{-\frac{t}{\tau^2}}J_0-\bar J\|_{L^\infty_tL^2_x}\leq C\tau.
\end{equation}
Moreover, by taking into account the expansion upto order $1$ in $\tau$, the asymptotic expansion of $\rho_\tau$ satisfies
\begin{equation}\label{eq:main_rho}
\|\rho_\tau-\bar\rho-\tau [e^{-s}\diver J_0+\rho_{\mathcal{O},1}]\|_{L^\infty_tL^2_{x}}\leq C\tau^2.
\end{equation}
\end{thm}

\begin{rem}
By the inner expansion, the leading term of the initial layer is given by
\[
J_\tau-\bar J=\tau^{-1}e^{-\frac{t}{\tau^2}}J_0+\dots,
\]
therefore the convergence of $J_\tau$ to $\bar J$ only holds in $L^p_t$ for $1\leq p<2$, which is consistent with the result in the authors' previous work \cite{AMZ3}. The subtraction of the initial-layer correction $\tau^{-1}e^{-\frac{t}{\tau^2}}J_0$ is a necessary prerequisite for obtaining the strong convergence of $J_\tau$ in $L^p_t$, where $2\le p\leq\infty$. 
\end{rem}

\subsection{Structure of the paper and notations}

This paper is structured as follows. By using the method of matching asymptotic expansions, we analyze the initial layer to obtain the equations of $(\rho_{\mathcal{I},k},J_{\mathcal{I},k})$ at each order and solve the explicit solutions in Section \ref{sect:inner}, and Section \ref{sect:outer} is devoted to the analysis of the outer functions $(\rho_{\mathcal{O},k},J_{\mathcal{O},k})$. In Section \ref{sect:remainder}, we compute the equations of the remainders $(r,\mathcal{R})$ by summarizing the previous results of the initial layer and outer functions. Last, the uniform energy estimate of $(r,\mathcal{R})$ is obtained in Section \ref{sect:en_remainder}.

In this paper, Lebesgue and Sobolev norms on ${\T^d}$ are denoted by
\[
||f||_{L_{x}^{p}}\coloneqq(\int_{{\T^d}}|f(x)|^{p}dx)^{\frac{1}{p}},
\]
\[
||f||_{W_{x}^{k,p}}\coloneqq\sum_{j=0}^k||\partial_{x}^{j}f||_{L_{x}^{p}},
\]
and $H_{x}^{k}\coloneqq H^{k}({\T^d})$ denotes the Sobolev space $W^{k,2}({\T^d})$.
Given a time interval $I\subset[0, \infty)$, the mixed space-time Lebesgue norm of $f:I\to L^{r}({\T^d})$ is
given by
\[
||f||_{L_{t}^{q}L_{x}^{r}}\coloneqq\left(\int_{I}||f(t)||_{L_{x}^{r}}^{q}dt\right)^{\frac{1}{q}}=\left(\int_{I}\left(\int_{{\T^d}}|f(x)|^{r}dx\right)^{\frac{q}{r}}dt\right)^{\frac{1}{q}}.
\]
Similarly the mixed Sobolev norm $L_{t}^{q}W_{x}^{k,r}$ is defined. We use $C$ to denote a generic constant, which may changes from line to line, $C(A)$ indicates its dependence on the quantity $A$.

\section{Inner expansion of the initial layer}\label{sect:inner}

In this section we first compute the equations of the initial layer functions \eqref{eq:inner} by considering the asymptotic expansion of QHD system \eqref{eq:QHD_rs_intro} with respect to the relaxation time $\tau$. 

We first compute the expansion of the nonlinearity in system \eqref{eq:QHD_rs_intro}. 
For convenience of notation, we use \eqref{eq:bohm} to define the consistent momentum operator as
\begin{equation}\label{eq:opbarJ}
\bar J(f)=\frac14\nabla\triangle f-\frac14\diver\left(\frac{\nabla f\otimes\nabla f}{f}\right)-\nabla p(f)-f\nabla V(f),
\end{equation}
where
\[
V(f)=(-\triangle)^{-1}f
\]
for functions $f\in L^1_x({\T^d})$ such that $\int_{\T^d} f\,dx=0$.  
Then to compute the expansion of the nonlinear operator $\bar J(f_0+f)$, we use the Taylor theorem
\[
(f_0+f)^{-1}=f_0^{-1}-\frac{f}{f_0^2}+2\int_0^f\frac{f-g}{(f_0+g)^3}dg,
\]
and
\[
p(f_0+f)=p(f_0)+p'(f_0)f+\int_0^f p''(f_0+g)(f-g)dg,
\]
Thus by letting 
\[
\rho=\rho_0+\tau\rho_{1}+\tau^2 r=\rho_0+\tau \eta_\tau,
\]
we expand the quadratic term in $\bar J(\rho)$ upto order $\tau^2$ as
\begin{align*}
\frac{\nabla\rho\otimes\nabla\rho}{4\rho}
=&\frac{\nabla\rho_0\otimes \nabla\rho_0}{4\rho_0}-\tau\rho_1\frac{\nabla\rho_0\otimes \nabla\rho_0}{4\rho_0}+\tau\frac{(\nabla\rho_0\otimes \nabla\rho_1)^s}{2\rho_0}\\
&-\tau^2 r\frac{\nabla\rho_0\otimes \nabla\rho_0}{4\rho_0}-\tau^2\eta_\tau\frac{(\nabla\rho_0\otimes \nabla\rho_1)^s}{2\rho_0}\\
&+\tau^2\frac{(\nabla\rho_0\otimes \nabla r)^s}{2\rho}+\tau^2\frac{\nabla\eta_\tau\otimes \nabla\eta_\tau}{2\rho},
\end{align*}
where $A^s=\frac12(A+A^T)$ denotes the symmetric part of matrix $A\in\R^{d\times d}$
Then we obtain the consistent momentum density as
\begin{equation}\label{eq:expdbarJ}
\bar J(\rho)=\bar J(\rho_0)+\tau \mathcal{A}(\rho_0,\rho_1)+\tau^2 \mathcal{B}(\rho).
\end{equation}
Operators $\mathcal{A}$ and $\mathcal{B}$ are defined by
\begin{equation}\label{eq:opA}
\begin{aligned}
\mathcal{A}(\rho_0,\rho_1)=&\frac14\nabla\triangle\rho_1-\frac12\diver\left[\frac{\nabla \rho_0\otimes\nabla\rho_1}{\rho_0}\right]^s+\frac14\diver\left[\frac{\rho_1\nabla \rho_0\otimes\nabla\rho_0}{\rho_0^2}\right]\\
&-\nabla[p'(\rho_0)\rho_1]-\rho_0\nabla V(\rho_1)-\rho_1\nabla V(\rho_0-\mathcal{C}(x)),
\end{aligned}
\end{equation}
and
\begin{equation}\label{eq:opB}
\begin{aligned}
\mathcal{B}(\rho)=&\frac14\nabla\triangle r-\frac14\diver\left[\frac{2(\nabla\rho_0\otimes \nabla r)^s+\nabla \eta_\tau\otimes \nabla\eta_\tau}{\rho}\right]+\frac14\diver\left[\frac{r\nabla \rho_0\otimes\nabla\rho_0}{\rho_0^2}\right]\\
&-\frac{1}{2\tau^2}\diver\left[(\nabla \rho_0\otimes\nabla\rho_0) \int_0^{\tau \eta_\tau}\frac{\tau \eta_\tau-g}{(\rho_0+g)^3}dg\right]+\frac{1}{2}\diver\left[\frac{\eta_\tau(\nabla \rho_0\otimes\nabla\rho_1)^s}{\rho_0^2}\right]\\
&-\frac{1}{\tau}\diver\left[(\nabla\rho_0\otimes\nabla\rho_1)^s\int_0^{\tau \eta_\tau}\frac{\tau \eta_\tau-g}{(\rho_0+g)^3}dg\right]\\
&-\nabla[p'(\rho_0)r]-\tau^{-2}\nabla\int_0^{\tau \eta_\tau} p''(\rho_0+g)(\tau \eta_\tau-g)dg\\
&-r\nabla V(\rho_0)-\rho_0\nabla V(r)-\eta_\tau\nabla V(\eta_\tau).
\end{aligned}
\end{equation}
Also by letting
\[
J=\tau^{-1}J_{-1}+J_0+\tau J_1+\tau^2\mathcal{R},\quad\xi_\tau=J_1+\tau\mathcal{R}.
\]
we write the quadratic term $J\otimes J/\rho$ as
\begin{equation}\label{eq:expquadJ}
\begin{aligned}
\frac{J\otimes J}{\rho}=&\tau^{-2}\frac{J_{-1}\otimes J_{-1}}{\rho_{0}}+2\tau^{-1}\frac{(J_{-1}\otimes J_{0})^s}{\rho_{0}}\\
&-\tau^{-1}\rho_{1}\frac{J_{-1}\otimes J_{-1}}{\rho_{0}^2}+\mathcal{Q}(\rho,J),
\end{aligned}
\end{equation}
where
\begin{align*}
\mathcal{Q}(\rho,J)=&\frac{J_{0}\otimes J_{0}+2[(J_{-1}+\tau J_{0})\otimes \xi_\tau ]^s+\tau^2\xi_\tau \otimes \xi_\tau }{\rho }\\
&-r\frac{J_{-1}\otimes J_{-1}}{\rho_{0}^2}+2\tau^{-2}(J_{-1}\otimes J_{-1})\int_0^{\tau \eta_{\tau}}\frac{\tau \eta_{\tau}-g}{(\rho_0+g)^3}dg\\
&-2\frac{(J_{-1}\otimes J_{0})^s}{\rho_{0}^2}\eta_{\tau}+4\tau^{-1}(J_{-1}\otimes J_{0})^s\int_0^{\tau \eta_{\tau}}\frac{\tau \eta_{\tau}-g}{(\rho_0+g)^3}dg,
\end{align*}

Now let us consider a solution $(\rho_\tau,J_\tau)$ of \eqref{eq:QHD_rs_intro}, with initial data $(\rho_0,\tau^{-1}J_0)$, in the form of expansion \eqref{eq:exptau}. To analyze the solution in the initial layer, we need to transfer the outer functions $f_{\mathcal{O}}$  ($=\rho_{\mathcal{O}}$ or $J_{\mathcal{O}}$) into fast time variable $s=\tau^{-n}t$,  
\begin{align*}
f_{\mathcal{O}}(t,x)=f_{\mathcal{O}}(\tau^n s,x)=\sum_k \tau^k f_{\mathcal{O},k}(\tau^n s,x)
\end{align*}
By Taylor expansion, we further write $f_{\mathcal{O},k}$ as
\[
f_{\mathcal{O},k}(\tau^n s,x)=\sum_{j=0}^\infty \tau^{nj}s^j\left.\d_t^j f_{\mathcal{O},k}\right|_{t=0}.
\]
Then by re-organizing terms in the order of $\tau$, we have
\[
f_{\mathcal{O}}(\tau^n s,x)=\sum_k \tau^k P_{f,k}(s,x)
\]
where
\begin{equation}\label{eq:match}
    P_{f,k}(s,x)=\sum_{j=0}^{\lfloor k/n\rfloor}\frac{s^j}{j!}\left.\d_t^j f_{\mathcal{O},k-jn}\right|_{t=0}.
\end{equation}
Thus formally we have
\[
\sum_{k\leq 1}\left[f_{\mathcal{O},k}(t,x)-P_{f,k}(s,x)\right]=\sum_{k=2}^\infty \tau^k P_{f,k}(s,x)=\mathcal{O}(\tau^2),
\]
and in the initial layer, we can write $(\rho_\tau,J_\tau)$ as
\[
\begin{aligned}
\rho_\tau(s,t,x)=\sum_{k=0}^1\rho_{\mathcal{I},k}\left(s,x\right)+\tau^2 r_\mathcal{I}(s,t,x)\\
J_\tau(s,t,x)= \sum_{k=-1}^1 J_{\mathcal{I},k}\left(s,x\right)+\tau^2\mathcal{R}_\mathcal{I}(s,t,x),
\end{aligned}
\]
By substitute it into system \eqref{eq:QHD_rs_intro} and using the expansion \eqref{eq:expdbarJ} and \eqref{eq:expquadJ}, it follows that
\begin{equation}\label{eq:inner_rho}
\begin{aligned}
0=&\tau^{- n}\d_s\rho_{\mathcal{I},0}+\tau^{1- n}\d_s\rho_{\mathcal{I},1}+\tau^{2- n}\d_s r_{\mathcal{I}}\\
&+\tau^{-1}\diver J_{\mathcal{I},-1}+\diver J_{\mathcal{I},0}+\tau\diver J_{\mathcal{I},1}+\tau^{2}\diver \mathcal{R}_\mathcal{I},
\end{aligned}
\end{equation}
and
\begin{equation}\label{eq:inner_J}
\begin{aligned}
0=&\tau^{1- n}\d_sJ_{\mathcal{I},-1}+\tau^{-1}J_{\mathcal{I},-1}\\
&\tau^{2- n}\d_sJ_{\mathcal{I},0}+\diver\left(\frac{J_{\mathcal{I},-1}\otimes J_{\mathcal{I},-1}}{\rho_{\mathcal{I},0}}\right)-\bar J(\rho_{\mathcal{I},0})+J_{\mathcal{I},0}\\
&\tau^{3- n}\d_s J_{\mathcal{I},1}+\tau\diver\left[2\frac{(J_{\mathcal{I},-1}\otimes J_{\mathcal{I},0})^s}{\rho_{\mathcal{I},0}}-\frac{J_{\mathcal{I},-1}\otimes J_{\mathcal{I},-1}\rho_{\mathcal{I},1}}{\rho_{\mathcal{I},0}^2}\right]-\tau \mathcal{A}(\rho_{\mathcal{I},0},\rho_{\mathcal{I},1})+\tau J_{\mathcal{I},1}\\
&\tau^{4- n}\d_s\mathcal{R}_\mathcal{I}+\tau^2\diver \mathcal{Q}(\rho_\tau,J_\tau)-\tau^2 \mathcal{B}(\rho_\tau)+\tau^2\mathcal{R}_\mathcal{I}.
\end{aligned}
\end{equation}
By matching the order of $\tau$ in \eqref{eq:inner_J} (since the initial layer only exists in the momentum density), the time scale of the initial layer can be chosen as
\[
n=2.
\]
Then by separating terms with different order of $\tau$, we obtain the ODE systems of inner densities and inner momentum densities as
\begin{equation}\label{eq:innereq_rho}
\left\{\begin{aligned}
&\d_s\rho_{\mathcal{I},0}=0\\
&\d_s\rho_{\mathcal{I},1}+\diver J_{\mathcal{I},-1}=0
\end{aligned}\right.
\end{equation}
and
\begin{equation}\label{eq:innereq_J}
\left\{\begin{aligned}
&\d_sJ_{\mathcal{I},-1}+J_{\mathcal{I},-1}=0\\
&\d_sJ_{\mathcal{I},0}+\diver\left(\frac{J_{\mathcal{I},-1}\otimes J_{\mathcal{I},-1}}{\rho_{\mathcal{I},0}}\right)-\bar J(\rho_{\mathcal{I},0})+J_{\mathcal{I},0}=0\\
&\d_s J_{\mathcal{I},1}+\diver \left[2\frac{(J_{\mathcal{I},-1}\otimes J_{\mathcal{I},0})^s}{\rho_{\mathcal{I},0}}-\frac{J_{\mathcal{I},-1}\otimes J_{\mathcal{I},-1}\rho_{\mathcal{I},1}}{\rho_{\mathcal{I},0}^2}\right]-\mathcal{A}(\rho_{\mathcal{I},0},\rho_{\mathcal{I},1})+J_{\mathcal{I},1}=0
\end{aligned}\right.
\end{equation}
The initial data of the inner functions are chosen to match the expansion of the original initial density and momentum density, where in our case we have fixed initial density $\rho_0$ and the initial momentum density is given by $\tau^{-1}J_0$. Therefore we set
\begin{equation}
\rho_{\mathcal{I},0}(0)=\rho_0,\quad \rho_{\mathcal{I},1}(0)=0.
\end{equation}
and
\begin{equation}
J_{\mathcal{I},-1}(0)=J_0,\quad J_{\mathcal{I},0}(0)=J_{\mathcal{I},1}(0)=0.
\end{equation}

Now we can solve the ODE systems \eqref{eq:innereq_rho} and \eqref{eq:innereq_J}. It is straightforward that
\begin{equation}\label{eq:I_-1}
\rho_{\mathcal{I},0}=\rho_0,\quad J_{\mathcal{I},-1}(s)=e^{-s}J_0.
\end{equation}
Substituting \eqref{eq:I_-1} into the second equation of \eqref{eq:innereq_rho} and \eqref{eq:innereq_J}, we obtain
\begin{equation}\label{eq:rho_I1}
\rho_{\mathcal{I},1}=(e^{-s}-1)\diver J_0
\end{equation}
and
\[
\d_sJ_{\mathcal{I},0}+e^{-2s}\diver\left(\frac{J_0\otimes J_0}{\rho_0}\right)-\bar J(\rho_0)+J_{\mathcal{I},0}=0.
\]
Thus $J_{\mathcal{I},0}$ is solved by
\begin{equation}\label{eq:J_I0}
\begin{aligned}
J_{\mathcal{I},0}(s)=&\int_0^s e^{s_1-s}\left[\bar J(\rho_0)-e^{-2s_1}\diver\left(\frac{J_0\otimes J_0}{\rho_0}\right)\right]ds_1\\
=&(1-e^{-s})\left[\bar J(\rho_0)-e^{-s}\diver\left(\frac{J_0\otimes J_0}{\rho_0}\right)\right].
\end{aligned}
\end{equation}
Inductively, $J_{\mathcal{I},1}$ is solved by
\begin{align*}
J_{\mathcal{I},1}(s)=&\int_0^s e^{s_1-s}\mathcal{A}(\rho_0,\rho_{\mathcal{I},1})(s_1)ds_1\\
&+e^{-s}\int_0^s e^{-s_1}(e^{-s_1}-1)\diver\left(\frac{J_0\otimes J_0\diver J_0}{\rho_0^2} \right)ds_1\\
&-2e^{-s}\int_0^s(1-e^{-s_1})\diver\left[\frac{(J_0\otimes\bar J(\rho_0))^s}{\rho_0}-e^{-s_1}\left(\frac{J_0}{\rho_0}\otimes\diver\left(\frac{J_0\otimes J_0}{\rho_0}\right)\right)^s\right]ds_1
\end{align*}
where by \eqref{eq:expdbarJ} and previous results we have
\begin{align*}
\mathcal{A}(\rho_0,\rho_{\mathcal{I},1})(s_1)=(e^{-s_1}-1)\mathcal{A}(\rho_0,\diver J_0).
\end{align*}
Thus we obtain
\begin{equation}\label{eq:J_I1}
\begin{aligned}
J_{\mathcal{I},1}(s)=&[e^{-s}(s+1)-1]\mathcal{A}(\rho_0,\diver J_0)+2e^{-s}(1-e^{-s}-s)\diver \left[\frac{J_0\otimes\bar J(\rho_0)}{\rho_0}\right]^s\\
&+e^{-s}(1+e^{-2s}-2e^{-s})\diver\left\{\left[\left(\frac{J_0}{\rho_0}\otimes\diver\left(\frac{J_0\otimes J_0}{\rho_0}\right)\right)^s\right]-\frac{J_0\otimes J_0\diver J_0}{2\rho_0^2}\right\}.
\end{aligned}
\end{equation}

With the explicit solutions of $(\rho_{\mathcal{I},k},J_{\mathcal{I},k})$, we can compute the limit
\[
\lim_{s\to\infty}(\rho_{\mathcal{I},k},J_{\mathcal{I},k})(s),\quad k=0,1,
\]
which will play essential role in the determination of the matching condition \eqref{eq:match_intro} for the outer functions. For the inner densities, we have
\begin{equation}\label{eq:infty_in_dens}
\lim_{s\to\infty}\rho_{\mathcal{I},0}(s)=\rho_0,\quad \lim_{s\to\infty}\rho_{\mathcal{I},1}(s)=-\diver J_0,
\end{equation}
and for the inner momentum density, we have
\begin{equation}\label{eq:infty_in_mom}
\lim_{s\to\infty}J_{\mathcal{I},0}(s)=\bar J(\rho_0),\quad \lim_{s\to\infty}J_{\mathcal{I},1}(s)=-\mathcal{A}(\rho_0,\diver J_0).
\end{equation}

\begin{rem}\label{rem:wellpre}
The result of \cite{LT} indicates that for well-prepared initial data defined by \eqref{eq:cond_wellpre}, the convergence rate of the time-relaxation limit can be improved to $\tau^2$. We will provide a structural explanation of this improvement of convergence rate by using our asymptotic expansion.

Actually, we can relax the requirement of well-prepared initial data to the follows. We say the initial momentum density $J_0(\cdot;\tau)$ is decaying if
\begin{equation}\label{eq:cond_decay}
J_0(\cdot;\tau)=\tau J^*(\cdot).
\end{equation}
In the case of decaying initial momentum density, the inner functions of initial layer are solved by
\[
\rho_{\mathcal{I},0}=\rho_0,\quad J_{\mathcal{I},-1}(s)=0,
\] 
and
\[
\rho_{\mathcal{I},1}=0,\quad J_{\mathcal{I},0}(s)=e^{-s}J^*+(1-e^{-s})\bar J(\rho_0),\quad J_{\mathcal{I},1}(s)=0.
\]
Moreover, if the initial data is well-prepared, namely $J^*=\bar J(\rho_0)$, it follows that
\[
J_{\mathcal{I},0}(s)\equiv \bar J(\rho_0),
\]
which means there exists no initial layer for such well-prepared initial momentum. We will further see the influence of such initial data on the outer functions \eqref{eq:outer} in the next section.  
\end{rem}

\section{Outer expansion and matching condition}\label{sect:outer}

In this section we will obtain the equations of the outer functions \eqref{eq:outer} following similar argument as Section \ref{sect:inner}. Moreover, we also need to determine the initial data of the outer equations, such that the outer functions connect with the initial layer in a consistent way. Let us recall the asymptotic approximations of the form
\begin{align*}
\rho_\tau(t,x)=\rho_\mathcal{O}(t,x)+\rho_\mathcal{I}(s,x)-P_\rho(s,x)+\tau^2 r(s,t,x)\\
J_\tau(t,x)= J_\mathcal{O}(t,x)+J_\mathcal{I}(s,x)-P_J(s,x)+\tau^2\mathcal{R}(s,t,x),
\end{align*}
where the matching functions $P_f=\sum_kP_{f,k}(s,x)$ ($f=\rho$ or $J$) are given by \eqref{eq:match} and are determined by the initial data of the outer functions. As shown in the last section, the inner functions $(\rho_{\mathcal{I},k},J_{\mathcal{I},k})$, $k=0,1$, decay exponentially to end states as $s\to\infty$, see formula \eqref{eq:I_-1} to \eqref{eq:J_I1}, therefore we should set the initial data of the outer functions such that the following matching condition holds,
\begin{equation}\label{eq:cond_match}
\lim_{s\to\infty}(\rho_{\mathcal{I},k},J_{\mathcal{I},k})(s)=\lim_{s\to\infty}(P_{\rho,k},P_{J,k})(s),\quad k=0,1.
\end{equation}
This condition guarantees that $\rho_{\mathcal{I},k}(s)-P_{\rho,k}(s)$ decays exponentially to $0$ in $s=\tau^{-2}t$, which makes this term a higher order remainder when we consider the outer functions, namely
\begin{align*}
\rho_\tau(s,t,x)=\sum_{k=0}^1\rho_{\mathcal{O},k}\left(t,x\right)+\tau^2 r_\mathcal{O}(s,t,x)\\
J_\tau(s,t,x)= \sum_{k=0}^1 J_{\mathcal{O},k}\left(t,x\right)+\tau^2\mathcal{R}_\mathcal{O}(s,t,x).
\end{align*}
By using computation \eqref{eq:infty_in_dens} and \eqref{eq:infty_in_mom} from the last section, we can explicitly write down the matching condition \eqref{eq:cond_match} as
\begin{equation}\label{eq:ini_outer}
\begin{aligned}
&P_{\rho,0}(x)=\rho_{\mathcal{O},0}(0,x)=\lim_{s\to\infty}\rho_{\mathcal{I},0}(s)=\rho_0,\\
&P_{\rho,1}(x)=\rho_{\mathcal{O},1}(0,x)=\lim_{s\to\infty}\rho_{\mathcal{I},1}(s)=-\diver J_0,\\
&P_{J,0}(x)=J_{\mathcal{O},0}(0,x)=\lim_{s\to\infty}J_{\mathcal{I},0}(s)=\bar{J}(\rho_0),\\
&P_{J,1}(x)=J_{\mathcal{O},1}(0,x)=\lim_{s\to\infty}J_{\mathcal{I},1}(s)=-\mathcal{A}(\rho_0,\diver J_0),
\end{aligned}
\end{equation}

Now we can compute the equations of the outer functions. Similar as the inner function , we consider an outer solution of \ref{eq:QHD_rs_intro} given by the form
\begin{align*}
\rho_\tau=\rho_{\mathcal{O},0}(t,x)+\tau\rho_{\mathcal{O},1}(t,x)+\tau^2 r_{\mathcal{O}}(s,t,x),\\
J_\tau=J_{\mathcal{O},0}(t,x)+\tau J_{\mathcal{O},1}(t,x)+\tau^2 \mathcal{R}_{\mathcal{O}}(s,t,x).
\end{align*}
We again use expansion \eqref{eq:expdbarJ} of $\bar J(\rho_\tau)$, and substitute it into equation \eqref{eq:QHD_rs_intro}. Then we obtain
\begin{equation}\label{eq:outer_rho}
\begin{aligned}
0=&\d_t\rho_{\mathcal{O},0}+\tau^{1}\d_s\rho_{\mathcal{O},1}+\tau^{2}\d_s r_{\mathcal{O}}\\
&+\diver J_{\mathcal{O},0}+\tau\diver J_{\mathcal{O},1}+\tau^{2}\diver \mathcal{R}_\mathcal{O},
\end{aligned}
\end{equation}
and
\begin{equation}\label{eq:outer_J}
\begin{aligned}
0=&\bar J(\rho_{\mathcal{O},0})-J_{\mathcal{O},0}+\tau \mathcal{A}(\rho_{\mathcal{O},0},\rho_{\mathcal{O},1})-\tau J_{\mathcal{O},1}\\
&\tau^2\mathcal{B}(\rho_\tau)-\tau^2\mathcal{R}_\mathcal{O}-\tau^2\d_tJ_\tau-\tau^2\diver\left(\frac{J_\tau\otimes J_\tau}{\rho_\tau}\right).
\end{aligned}
\end{equation}
By separating the terms in \eqref{eq:outer_rho} and \eqref{eq:outer_J} with respect to the order of $\tau$, we obtain the $0$th-order equation as
\begin{equation}\label{eq:0equ}
\d_t\rho_{\mathcal{O},0}+\diver J_{\mathcal{O},0}=0,\quad J_{\mathcal{O},0}=\bar J(\rho_{\mathcal{O},0})
\end{equation}
and the $1$st-order equation as
\begin{equation}\label{eq:1equ}
\d_t\rho_{\mathcal{O},1}+\diver J_{\mathcal{O},1}=0,\quad J_{\mathcal{O},1}=\mathcal{A}(\rho_{\mathcal{O},0},\rho_{\mathcal{O},1}),
\end{equation}
which is exactly the linearization of \eqref{eq:0equ} around $\rho_{\mathcal{O},0}$. 
By \eqref{eq:opA}, the equation of $\rho_{\mathcal{O},1}$ given by \eqref{eq:1equ} is explicitly written as a $4-$th order parabolic equation
\begin{equation}\label{eq:rho1}
\begin{aligned}
0=&\d_t\rho_{\mathcal{O},1}+\frac14\triangle^2\rho_{\mathcal{O},1}-\frac12\diver\cdot\diver\left[\frac{\nabla \bar\rho\otimes\nabla\rho_{\mathcal{O},1}}{\bar\rho}\right]^s\\
&+\frac14\diver\cdot\diver\left[\frac{(\nabla \bar\rho\otimes\nabla\bar\rho)\rho_{\mathcal{O},1}}{\bar\rho^2}\right]-\triangle [p'(\bar\rho)\rho_{\mathcal{O},1}]\\
&-\diver[\rho_{\mathcal{O},1}\nabla V(\bar\rho-\mathcal{C}(x))]-\diver[\bar\rho\nabla V(\rho_{\mathcal{O},1})],
\end{aligned}
\end{equation}
which is linear in $\rho_{\mathcal{O},1}$. Combining with the initial data given by \eqref{eq:ini_outer}, the equations of $(\rho_{\mathcal{O},0},J_{\mathcal{O},0})$ show that $(\rho_{\mathcal{O},0},J_{\mathcal{O},0})=(\bar\rho,\bar J)$ is exactly the solution of the QDD equation \eqref{eq:qdde}. The well-posedness result of equation \eqref{eq:qdde} with $\bar\rho\geq\delta>0$ is given by \cite{AMZ3}. Global well-posedness of equation \eqref{eq:rho1} directly follows the theory of analytic semigroups of parabolic equations, see for example \cite{St} and Chapter 3 of \cite{Lu}.

\begin{rem}[Well-prepared initial data]
Let us continue the argument of Remark \ref{rem:wellpre} to exhibit the consequence of decaying and well-prepared initial momentum density. For such data as defined in Remark \ref{rem:wellpre}, the matching conditions become
\begin{align*}
&P_{\rho,0}(x)=\rho_{\mathcal{O},0}(0,x)=\lim_{s\to\infty}\rho_{\mathcal{I},0}(s)=\rho_0,\\
&P_{\rho,1}(x)=\rho_{\mathcal{O},1}(0,x)=\lim_{s\to\infty}\rho_{\mathcal{I},1}(s)=0,\\
&P_{J,0}(x)=J_{\mathcal{O},0}(0,x)=\lim_{s\to\infty}J_{\mathcal{I},0}(s)=\bar{J}(\rho_0),\\
&P_{J,1}(x)=J_{\mathcal{O},1}(0,x)=\lim_{s\to\infty}J_{\mathcal{I},1}(s)=0.
\end{align*}
In this case, we see that the zero-order outer equation \eqref{eq:0equ} is still the QDD equation \eqref{eq:qdde}, which is solved by $(\bar\rho,\bar J)$. The initial data of the first-order equation \eqref{eq:1equ} vanishes, therefore $(\rho_{\mathcal{O},1},J_{\mathcal{O},1})=(0,0)$, which explain the result of \cite{LT} that for decaying and well-prepared initial momentum, the convergence rate of the relaxation-time limit is possible to be improved to $\tau^2$. 
\end{rem}

\section{Approximating solution and equations of the remainders}\label{sect:remainder}

Recall from \eqref{eq:exptau}, the approximating solution $(\tilde \rho_\tau,\tilde J_\tau)$ is defined as
\begin{equation}\label{eq:aprx}
\begin{aligned}
\tilde \rho_\tau(s,t)=&\sum_{k=0}^1 \tau^k\left[(\rho_{\mathcal{I},k}(s)-P_{\rho,k})+\rho_{\mathcal{O},k}(t)\right],\\
\tilde J_\tau(s,t)=&\tau^{-1}J_{\mathcal{I},-1}(s)+\sum_{k=0}^1 \tau^k\left[(J_{\mathcal{I},k}(s)-P_{J,k})+J_{\mathcal{O},k}(t)\right].
\end{aligned}
\end{equation}
By summarizing the results in previous sections, we can explicitly write the approximating solution as
\begin{equation}\label{eq:rho_aprx}
\tilde \rho_\tau(s,t)=\bar\rho(t)+\tau [e^{-s}\diver J_0+\rho_{\mathcal{O},1}(t)],
\end{equation}
and
\begin{equation}\label{eq:J_aprx}
\begin{aligned}
\tilde J_\tau(s,t)=&\tau^{-1}e^{-s}J_0-e^{-s}\left[\bar J(\rho_0)+(1-e^{-s})\diver\left(\frac{J_0\otimes J_0}{\rho_0}\right)\right]+\bar J(t)\\
&+2\tau e^{-s}(1-e^{-s}-s)\diver\left[\frac{J_0\otimes \bar J(\rho_0)}{\rho_0}\right]^s\\
&+2\tau e^{-s}(1+e^{-2s}-2e^{-s})\diver\left\{\left[\frac{J_0}{\rho_0}\otimes\diver\left(\frac{J_0\otimes J_0}{\rho_0}\right)\right]^s-\frac{J_0\otimes J_0\diver J_0}{2\rho_0^2}\right\}\\
&-\tau e^{-s}(s+1)\mathcal{A}(\rho_0,\diver J_0)+\tau J_{\mathcal{O},1}(t)
\end{aligned}
\end{equation}
where the fast time variable $s=\tau^{-2}t$. 

The remainder functions are defined as
\begin{equation}\label{eq:r_1}
(r,\mathcal{R})(s,t)=\tau^{-2}(\rho_\tau-\tilde \rho_\tau,J_\tau-\tilde J_\tau)(s,t),
\end{equation}
and the main purpose of this section is to obtain the equations of $(r,\mathcal{R})$ by using the previous results. 

\begin{prop}
Remainders $(r,\mathcal{R})$ satisfy the following equations,
\begin{equation}\label{eq:dtr}
\d_t r=-\sum_{k=0}^1\tau^{k-2} \diver [J_{\mathcal{I},k}-P_{J,k}](s)-\diver\mathcal{R}.
\end{equation}
and
\begin{equation}\label{eq:dtR}
\begin{aligned}
\tau^2\d_t\mathcal{R}+\sum_{k=0}^1 \tau^{k}\d_t J_{\mathcal{O},k}+\diver\left(\frac{J_\tau\otimes J_\tau}{\rho_\tau}\right)=&\tau^{-1} e^{-s}[\mathcal{A}(\bar\rho,\diver J_0)-\mathcal{A}(\rho_0,\diver J_0)]\\
&+\mathcal{B}(\bar\rho,e^{-s}\diver J_0+\rho_{\mathcal{O},1},r)+\mathcal{G}_1(s,x)-\mathcal{R},
\end{aligned}
\end{equation}
where the operator $\mathcal{A}$, $\mathcal{B}$ are defined by \eqref{eq:opA}, \eqref{eq:opB}, and the function $\mathcal{G}_1$ is given by
\begin{align*}
\mathcal{G}_1(s,x)=&\tau^{-2}e^{-2s}\diver \left(\frac{J_0\otimes J_0}{\rho_0}\right)\\
&+\tau^{-1}2e^{-s}(1-e^{-s})\diver \left[\frac{J_0}{\rho_0}\otimes\left(\bar J(\rho_0)-e^{-s}\diver \left(\frac{J_0\otimes J_0}{\rho_0}\right)\right)\right]^s\\
&-\tau^{-1}e^{-2s}(e^{-s}-1)\diver\left(\frac{J_0\otimes J_0\diver J_0}{\rho_0^2}\right)
\end{align*}
\end{prop}

\begin{rem}
The leading singular term in $\mathcal{G}_1$ has only $L^1_t$ uniform integrability, which is not sufficient for the $\tau-$uniform estimate of $(r,\mathcal{R})$. Further analysis of such singular terms is provided in the next subsection.
\end{rem}

\begin{proof}
First, we have
\[
\d_t r=\tau^{-2}(\d_t\rho_\tau-\d_t\tilde \rho_\tau)=\tau^{-2}(-\diver J_\tau-\d_t\bar\rho+\tau^{-1}e^{-s}\diver J_0-\d_t\rho_{\mathcal{O},1}).
\]
By using \eqref{eq:aprx}, \eqref{eq:0equ} and \eqref{eq:1equ}, we obtain \eqref{eq:dtr}. 

Now we consider the equation of $\mathcal{R}$. By using \eqref{eq:aprx} and \eqref{eq:innereq_J}, we have
\begin{align*}
\tau^2\d_t J_\tau=&\tau^{-1}\d_s J_{\mathcal{I},-1}+\sum_{k=0}^1 \tau^k\left[\d_s J_{\mathcal{I},k}(s)+\tau^2\d_t J_{\mathcal{O},k}(t)\right]+\tau^4 \d_t\mathcal{R}\\
=&-\tau^{-1}J_{\mathcal{I},-1}-\diver\left(\frac{J_{\mathcal{I},-1}\otimes J_{\mathcal{I},-1}}{\rho_{\mathcal{I},0}}\right)+\bar J(\rho_{0})-J_{\mathcal{I},0}\\
&-\tau\diver\left[2\frac{(J_{\mathcal{I},-1}\otimes J_{\mathcal{I},0})^s}{\rho_{\mathcal{I},0}}-\frac{J_{\mathcal{I},-1}\otimes J_{\mathcal{I},-1}\rho_{\mathcal{I},1}}{\rho_{\mathcal{I},0}^2}\right]\\
&+\tau\mathcal{A}(\rho_{0},(e^{-s}-1)\diver J_0)-\tau J_{\mathcal{I},1}+\sum_{k=0}^1 \tau^{k+2}\d_t J_{\mathcal{O},k}+\tau^4 \d_t\mathcal{R}.
\end{align*}
By writing
\[
\rho_\tau=\bar\rho+\tau\eta_\tau,\quad \eta_\tau=e^{-s}\diver J_0+\rho_{\mathcal{O},1}(t)+\tau r,
\]
and using the expansion \eqref{eq:expdbarJ}, we have
\[
\bar J(\rho_\tau)=\bar J(t)+\tau\mathcal{A}(\bar\rho,e^{-s}\diver J_0)+\tau\mathcal{A}(\bar\rho,\rho_{\mathcal{O},1})+\tau^2\mathcal{B}(\rho_\tau).
\]
Then by using \eqref{eq:0equ} and \eqref{eq:1equ}, the momentum equation of $\d_t J_\tau$ is written as
\begin{align*}
\tau^2\d_tJ_\tau+\tau^2\diver \left(\frac{J_\tau\otimes J_\tau}{\rho_\tau}\right)=&\bar J(t)+\tau\mathcal{A}(\bar\rho,e^{-s}\diver J_0)+\tau\mathcal{A}(\bar\rho,\rho_{\mathcal{O},1})\\
&+\tau^2\mathcal{B}(\bar\rho,e^{-s}\diver J_0+\rho_{\mathcal{O},1},r)\\
&-\tau^{-1}J_{\mathcal{I},-1}-[(J_{\mathcal{I},0}-\bar J(\rho_0))+J_{\mathcal{O},0}]\\
&-\tau[(J_{\mathcal{I},1}-\mathcal{A}(\rho_0,\diver J_0))+J_{\mathcal{O},1}]-\tau^2\mathcal{R}\\
=&\tau\mathcal{A}(\bar\rho,e^{-s}\diver J_0)+\tau^2\mathcal{B}(\bar\rho,e^{-s}\diver J_0+\rho_{\mathcal{O},1},r)-\tau^{-1}J_{\mathcal{I},-1}\\
&-(J_{\mathcal{I},0}-\bar J(\rho_0))-\tau (J_{\mathcal{I},1}-\mathcal{A}(\rho_0,\diver J_0))-\tau^2\mathcal{R}.
\end{align*}
Comparing the two expressions of $\tau^2\d_tJ_\tau$, we obtain the equation of $\mathcal{R}$ as \eqref{eq:dtR}.

\end{proof}

\subsection{Modified momentum remainder $\overline{\mathcal{R}}$ and the corresponding equations}

We see that equations \eqref{eq:dtr} and \eqref{eq:dtR} are formally similar to the rescaled QHD \eqref{eq:QHD_rs_intro}, therefore we can expect $(r,\mathcal{R})$ to satisfy some estimates analogous to those of $(\rho_\tau,J_\tau)$, for example the "energy" estimates
\[
(\nabla r,\tau\mathcal{R})\in L^\infty_tL^2_x,\quad \mathcal{R}\in L^2_{t,x}.
\]
Due to the existence of initial layer, we should pay more attention to the time-integrability of the source terms in  \eqref{eq:dtr} and \eqref{eq:dtR}. In particular in $\mathcal{G}_1$, the term with fast time multiplier $\tau^{-2}e^{-2s}$, $s=\tau^{-2}t$, only provides $L^1$ uniform integrability in $t$. Such integrability will not allow us to obtain an estimate of $\mathcal{R}$ uniform in $\tau$.  Therefore, we need to treat more carefully the trouble terms in \eqref{eq:dtR}.

Since our goal is to establish the $L^2_{t,x}$ estimate of $\mathcal{R}$ uniformly in $\tau$, to simplify the analysis it is more convenient to include all other terms that already have $L^2_{t,x}$ integrability in a new definition of the momentum remainder, namely
\begin{equation}\label{eq:alter_R}
\overline{\mathcal{R}}=\tau^{-1}(J_{\mathcal{I},1}-P_{J,1})(s)+\mathcal{R}.
\end{equation}

\begin{prop}
Equations of the density remainder $r$ and the modified momentum remainder $\overline{\mathcal{R}}$ are given by
\begin{equation}\label{eq:dtr2}
\d_t r-\diver \mathcal{S}_\rho+\diver\overline{\mathcal{R}}=0,
\end{equation}
and
\begin{equation}\label{eq:dtR2}
\begin{aligned}
\tau^2\d_t\overline{\mathcal{R}}+\tau^2\diver\left(\frac{J_\tau\otimes\overline{\mathcal{R}}}{\rho_\tau}\right)+&\tau^2\diver\left[\frac{\overline{\mathcal{R}}\otimes(\tau^{-1}e^{-s}J_0+\xi_\tau)}{\rho_\tau}\right]\\
=&\mathcal{S}_J+\mathcal{B}(\bar\rho,e^{-s}\diver J_0+\rho_{\mathcal{O},1},r)-\overline{\mathcal{R}},
\end{aligned}
\end{equation}
where
\begin{equation}\label{eq:S_rho}
\mathcal{S}_\rho=\tau^{-2}e^{-s}\left[\bar J(\rho_0)+(1-e^{-s})\diver\left(\frac{J_0\otimes J_0}{\rho_0}\right)\right].
\end{equation}
and
\begin{equation}\label{eq:S2}
\begin{aligned}
\mathcal{S}_J=&\tau^{-1} e^{-s}\mathcal{A}(\bar\rho,\diver J_0)+\tau^{-2}e^{-2s}\diver\left[(J_0\otimes J_0)(\rho_0^{-1}-\rho_\tau^{-1})\right]\\
&-\diver\left[\frac{2\tau^{-1}e^{-s}\xi_\tau\otimes J_0+\xi_\tau\otimes \xi_\tau}{\rho_\tau}\right]-\sum_{k=0}^1 \tau^{k}\d_t J_{\mathcal{O},k}.
\end{aligned}
\end{equation}
\end{prop}

\begin{proof}
By using the modified momentum remainder $\overline{\mathcal{R}}$, it is straightforward to deduce the equation of density remainder $r$ from \eqref{eq:dtr} as \eqref{eq:dtr2}.

The equation of $\overline{\mathcal{R}}$ is computed exactly as the one of $\mathcal{R}$, but ignoring all the terms related to $[J_{\mathcal{I},1}-P_{J,1}]$, thus we obtain
\begin{align*}
\tau^2\d_t\overline{\mathcal{R}}+\sum_{k=0}^1 \tau^{k}\d_t J_{\mathcal{O},k}&+\diver\left(\frac{J_\tau\otimes J_\tau}{\rho_\tau}\right)=\tau^{-1} e^{-s}\mathcal{A}(\bar\rho,\diver J_0)\\
&+\mathcal{B}(\bar\rho,e^{-s}\diver J_0+\rho_{\mathcal{O},1},r)
+\tau^{-2}e^{-2s}\diver\left(\frac{J_0\otimes J_0}{\rho_0}\right)-\overline{\mathcal{R}},
\end{align*}
It remains to discuss how to deal with the terms which only have $L^1_t$ integrability. For this purpose, by using \eqref{eq:aprx} we can rewrite
\[
J_\tau=\tau^{-1}e^{-s}J_0+\xi_\tau+\tau^2\overline{\mathcal{R}},
\]
where
\begin{equation}\label{eq:xi}
\xi_\tau=-e^{-s}\left[\bar J(\rho_0)+(1-e^{-s})\diver\left(\frac{J_0\otimes J_0}{\rho_0}\right)\right]+\bar J(t)+\tau J_{\mathcal{O},1}.
\end{equation}
Then we can write the advection term as
\begin{align*}
\diver\left(\frac{J_\tau\otimes J_\tau}{\rho_\tau}\right)=&\diver\left[\frac{J_\tau\otimes(\tau^{-1}e^{-s}J_0+\xi_\tau)}{\rho_\tau}\right]+\tau^2\diver\left[\frac{J_\tau\otimes\overline{\mathcal{R}}}{\rho_\tau}\right]\\
=&\tau^2\diver\left(\frac{J_\tau\otimes\overline{\mathcal{R}}}{\rho_\tau}\right)+\tau^2\diver\left[\frac{\overline{\mathcal{R}}\otimes(\tau^{-1}e^{-s}J_0+\xi_\tau)}{\rho_\tau}\right]\\
&+\tau^{-2}e^{-2s}\diver\left(\frac{J_0\otimes J_0}{\rho_\tau}\right)+\diver\left[\frac{2\tau^{-1}e^{-s}\xi_\tau\otimes J_0+\xi_\tau \otimes\xi_\tau}{\rho_\tau}\right],
\end{align*}
where in this paper we defined the divergence of a non-symmetric matrix $A=(A_{j,k})$ as
\[
\diver A=(\d_{x_j}A_{j,k}),
\]
and the Einstein's convention of summation is always applied over duplicated indexes. 
Thus the equation of $\overline{\mathcal{R}}$ is obtained as \eqref{eq:dtR2}.
\end{proof}

The force term $\mathcal{S}_\rho$ in \eqref{eq:dtr2} only has uniform $L^1$ integrability in time. However it sufficient to perform our analysis for $\nabla r$, since we expect $\nabla r\in L^\infty_t$. While $\mathcal{S}_J$ is expected to be $L^2$ integrable in time. Let us consider \cite{AMZ3} Theorem 2 (1.20) or \cite{JLM} Lemma 4.3, then we expect that $\d_t\rho_\tau=-\diver J_\tau\in L^2_tH^\beta_x$ ($\beta$ depending on the initial regularity). In the case of strictly positive mass densities  ($\rho_\tau\geq\delta>0$), it follows that
\[
|\rho_0^{-1}-\rho_\tau^{-1}|\leq\int_0^t \frac{|\d_t \rho_\tau|}{\rho_\tau^2}(t_1)dt_1\leq \frac{\sqrt{t}}{\delta^2}\|\d_t\rho_\tau\|_{L^2_t}.
\]
Hence the behavior of leading singularity is no longer $\tau^{-2}e^{-2t/\tau^2}$, but is improved into $\tau^{-2}e^{-2t/\tau^2}\sqrt{t}$, which is $L^2_t$-bounded uniformly with respect to $\tau>0$. Therefore $S_J\in L^2_tH^{\beta'}_x$ uniformly in $\tau$, under suitable regularity assumption on the initial data. Next section is devoted to implement the program of estimating $\mathcal{S}_J$.

\section{Energy Estimates of $(r,\overline{\mathcal{R}})$}\label{sect:en_remainder}

In this section we will establish the uniform energy estimates on  the density remainder $r$ and the alternative momentum remainder $\overline{\mathcal{R}}$, under sufficient regularity assumption on the initial data $(\rho_0,J_0)$ and solutions $(\rho_\tau,J_\tau,\bar\rho,\rho_{\mathcal{O},1})$. We underline again that in this section we will not provide new results concerning the existence of solutions with high order regularity, but we will confine ourselves to identify a framework where the initial layer analysis is possible. More precisely, we will make the following regularity assumptions with $m\geq 8$.

\begin{assumption}[Regularity assumption]\label{asmp:RA}
Assume that on the time interval $[0,T_0]$, the following conditions hold.
\begin{itemize}
\item[(1)] Positivity of densities: there exists a $\delta>0$ such that
\[
\inf_{t,x}\,\{\rho_0,\rho_\tau,\bar\rho\}\geq\delta >0.
\]

\item[(2)] Irrotationality: the initial velocity $v_0=\frac{J_0}{\rho_0}$ satisfies
\[
(\nabla v_0)^a=\frac12\left(\nabla v_0-\nabla v_0^T\right)=0.
\]

\item[(3)] Uniform Sobolev regularity: there exists a $M>0$ such that
\begin{align*}
\max&\left\{\|\rho_0\|_{H^{m+1}_x},\|J_0\|_{H^{m}_x},\|\rho_\tau\|_{L^\infty_x H^{m+1}_x},\right.\\
&\hspace{0.3cm}\left.\|J_\tau\|_{L^\infty_x H^{m}_x},\|\bar\rho\|_{L^\infty_x H^{m+1}_x},\|\rho_{\mathcal{O},1}\|_{L^\infty_x H^{m-1}_x}\right\}\leq M.
\end{align*}
\end{itemize}
\end{assumption}

\begin{rem}
Results similar to Assumption \ref{asmp:RA} for $(\rho_\tau,J_\tau)$ and $\bar\rho$ are proved in the literature, provided the positivity and Sobolev bounds assumptions on the initial data $(\rho_0,J_0)$ and various kinds of additional conditions. For example, under smallness assumption such results for $\rho_\tau$ were proven in \cite{JP,LM}, and the corresponding regularity of $\bar\rho$ was provided as a result of the time-relaxation limit \cite{JLM}.

\end{rem}

Given the regularity assumption, the main goal of this section is to obtain uniform energy estimate for  $(r,\overline{\mathcal{R}})$. It will be important to use the following normalized energy functional
\begin{equation}\label{eq:en_r}
E_\tau(r,\overline{\mathcal{R}})=\int_{\T^d} \frac{1}{2\rho_\tau}(\tau^2|\overline{\mathcal{R}}|^2+|\nabla r|^2)dx.
\end{equation}
The following proposition is conserved with the time derivative of $E_\tau(r,\overline{\mathcal{R}})$.

\begin{prop}
Provided regularity Assumption \ref{asmp:RA}, the energy functional $E_\tau(r,\overline{\mathcal{R}})$ can be differentiable in time along the trajectory, and we have 
\begin{equation}\label{eq:dtEr}
\begin{aligned}
\frac{d}{dt}E_\tau(r,\overline{\mathcal{R}})+\int_{{\T^d}} \frac{|\overline{\mathcal{R}}|^2}{\rho_\tau}dx=&\int_{{\T^d}} \frac{\overline{\mathcal{R}}}{\rho_\tau}\cdot\mathcal{S}_J(t,x)dx-\frac{\tau^2}{2}\int_{\T^d}\frac{|\overline{\mathcal{R}}|^2}{\rho_\tau^2}\diver(\tau^{-1}e^{-s}J_0+\xi_\tau)dx\\
&+\int_{\T^d}\frac{\overline{\mathcal{R}}}{\rho_\tau}\cdot\left[-\tau^{-1} e^{-s}\left(v_0\otimes\nabla\left(\frac{\rho_0}{\rho_\tau}\right)\right)^a+\mathcal{S}_{rot}\right]\cdot(\tau^{-1}e^{-s}J_0+\xi_\tau)dx\\
&-\int_{\T^d}\frac{|\nabla r|^2}{2\rho_\tau^2}\d_t\rho_\tau dx+\int_{\T^d} \frac{\nabla\diver \mathcal{S}_\rho}{\rho_\tau}\cdot\nabla rdx\\
&+\int_{\T^d}\frac{\overline{\mathcal{R}}}{\rho_\tau}\cdot[\mathcal{B}_2(r,\nabla r)+\mathcal{B}_3(\nabla r)]dx,
\end{aligned}
\end{equation}
where $\mathcal{S}_{rot}$, $\mathcal{S}_J$, $\xi_\tau$, $\mathcal{S}_\rho$, $\mathcal{B}_2(r,\nabla r)$ and $\mathcal{B}_3(\nabla r)$ are given by \eqref{eq:S_rho}, \eqref{eq:S2}, \eqref{eq:xi}, \eqref{eq:S_curl}, \eqref{eq:B2} and \eqref{eq:B3} respectively.
\end{prop}

\begin{proof}

We start with equation \eqref{eq:dtR2} by multiplying \eqref{eq:dtR2} by $\overline{\mathcal{R}}/\rho_\tau$. For the advection term, we notice that
\begin{align*}
\tau^2\int_{\T^d} \frac{\overline{\mathcal{R}}}{\rho_\tau}\cdot\diver\left(\frac{J_\tau\otimes\overline{\mathcal{R}}}{\rho_\tau}\right)dx=&\tau^2\int_{\T^d} \frac{|\overline{\mathcal{R}}|^2}{\rho_\tau^2}\diver J_\tau dx+\frac{\tau^2}{2}\int_{\T^d} J_\tau\cdot\nabla\left(\frac{|\overline{\mathcal{R}}|^2}{\rho_\tau^2}\right)dx\\
=&-\frac{\tau^2}{2}\int_{\T^d} \frac{|\overline{\mathcal{R}}|^2}{\rho_\tau^2}\d_t\rho_\tau dx
\end{align*}
then we have
\begin{align*}
&\frac{\tau^2}{2}\frac{d}{dt}\int_{\T^d} \frac{|\overline{\mathcal{R}}|^2}{\rho_\tau}dx+\int_{\T^d} \frac{|\overline{\mathcal{R}}|^2}{\rho_\tau}dx\\
=&-\tau^2\int_{\T^d} \frac{\overline{\mathcal{R}}}{\rho_\tau}\cdot\diver\left[\frac{\overline{\mathcal{R}}\otimes(\tau^{-1}e^{-s}J_0+\xi_\tau)}{\rho_\tau}\right]dx+\int_{\T^d} \frac{\overline{\mathcal{R}}}{\rho_\tau}\cdot\mathcal{S}_J(t,x)dx\\
&+\int_{\T^d} \frac{\overline{\mathcal{R}}}{\rho_\tau}\cdot\mathcal{B}(\bar\rho,e^{-s}\diver J_0+\rho_{\mathcal{O},1},r)dx.
\end{align*}

The integral
\[
-\tau^2\int_{\T^d} \frac{\overline{\mathcal{R}}}{\rho_\tau}\cdot\diver\left[\frac{\overline{\mathcal{R}}\otimes(\tau^{-1}e^{-s}J_0+\xi_\tau)}{\rho_\tau}\right]dx
\]
can be computed by integrating by parts as
\begin{align*}
=&\tau^2\int_{\T^d} \d_j\left(\frac{\overline{\mathcal{R}}_k}{\rho_\tau}\right)\frac{\overline{\mathcal{R}}_j\otimes(\tau^{-1}e^{-s}J_0+\xi_\tau)_k}{\rho_\tau}dx\\
=&\tau^2\int_{\T^d} \d_k\left(\frac{\overline{\mathcal{R}}_j}{\rho_\tau}\right)\frac{\overline{\mathcal{R}}_j(\tau^{-1}e^{-s}J_0+\xi_\tau)_k}{\rho_\tau}dx\\
&+\tau^2\int_{\T^d} \left[\d_j\left(\frac{\overline{\mathcal{R}}_k}{\rho_\tau}\right)-\d_k\left(\frac{\overline{\mathcal{R}}_j}{\rho_\tau}\right)\right]\frac{\overline{\mathcal{R}}_j(\tau^{-1}e^{-s}J_0+\xi_\tau)_k}{\rho_\tau}dx\\
=&\frac{\tau^2}{2}\int_{\T^d} \nabla\left(\frac{|\overline{\mathcal{R}}|^2}{\rho_\tau^2}\right)\cdot(\tau^{-1}e^{-s}J_0+\xi_\tau) dx+\tau^2\int_{\T^d} \left(\nabla\frac{\overline{\mathcal{R}}}{\rho_\tau}\right)^a:\frac{\overline{\mathcal{R}}\otimes(\tau^{-1}e^{-s}J_0+\xi_\tau)}{\rho_\tau}dx\\
=&-\frac{\tau^2}{2}\int_{\T^d}\frac{|\overline{\mathcal{R}}|^2}{\rho_\tau^2}\diver(\tau^{-1}e^{-s}J_0+\xi_\tau)dx+\tau^2\int_{\T^d} \left(\nabla\frac{\overline{\mathcal{R}}}{\rho_\tau}\right)^a:\frac{\overline{\mathcal{R}}\otimes(\tau^{-1}e^{-s}J_0+\xi_\tau)}{\rho_\tau}dx,
\end{align*}
where $A^a=\frac12(A-A^T)$ is the anti-symmetric part of matrix $A\in\R^{d\times d}$. 
To treat the integral of the curl-matrix, we recall that the irrotaitonality of the velocity is preserved by the trajectory of \eqref{eq:QHD_rs_intro}, namely 
\[
(\nabla v_\tau)^a=\left(\nabla\frac{J_\tau}{\rho_\tau}\right)^a=0,
\]
by definition of $\overline{\mathcal{R}}$ and \eqref{eq:J_aprx}, it follows that
\begin{align*}
0=(\nabla v_\tau)^a=&\tau^{-1} e^{-s}\left(\nabla\frac{J_0}{\rho_\tau}\right)^a-e^{-s}\left[\nabla\frac{\bar J(\rho_0)}{\rho_\tau}+\frac{(1-e^{-s})}{\rho_\tau}\nabla\diver\left(\frac{J_0\otimes J_0}{\rho_0}\right)\right]^a\\
&+\left(\nabla\frac{\bar J(t)}{\rho_\tau}\right)^a+\tau \left(\nabla\frac{J_{\mathcal{O},1}}{\rho_\tau}\right)^a+\tau^2 \left(\nabla\frac{\overline{\mathcal{R}}}{\rho_\tau}\right)^a.
\end{align*}
Therefore
\begin{align*}
\tau^2 \left(\nabla\frac{\overline{\mathcal{R}}}{\rho_\tau}\right)^a
=-\tau^{-1} e^{-s}\left[\nabla\left(v_0\frac{\rho_0}{\rho_\tau}\right)\right]^a+\mathcal{S}_{rot}(t,x),
\end{align*}
where 
\begin{equation}\label{eq:S_curl}
\begin{aligned}
\mathcal{S}_{rot}(t,x)=&e^{-s}\left[\nabla\frac{\bar J(\rho_0)}{\rho_\tau}+\frac{(1-e^{-s})}{\rho_\tau}\nabla\diver\left(\frac{J_0\otimes J_0}{\rho_0}\right)\right]^a\\
&-\left(\nabla\frac{\bar J(t)}{\rho_\tau}\right)-\tau \left(\nabla\frac{J_{\mathcal{O},1}}{\rho_\tau}\right)
\end{aligned}
\end{equation}
and $\mathcal{S}_{rot}$ is uniformly bounded in $L^\infty_t$. For the first term in the right hand side of $\tau^2 \left(\nabla\frac{\overline{\mathcal{R}}}{\rho_\tau}\right)^a$, we further notice that by the irrotationality of $v_0$, we have
\[
\left[\nabla\left(v_0\frac{\rho_0}{\rho_\tau}\right)\right]=\left(v_0\otimes\nabla\left(\frac{\rho_0}{\rho_\tau}\right)\right)^a,
\]
thus we obtain that
\begin{align*}
&\tau^2\int_{\T^d} \left(\nabla\frac{\overline{\mathcal{R}}}{\rho_\tau}\right)^a:\frac{\overline{\mathcal{R}}\otimes(\tau^{-1}e^{-s}J_0+\xi_\tau)}{\rho_\tau}dx\\
=&\int_{\T^d}\frac{\overline{\mathcal{R}}}{\rho_\tau}\cdot\left[-\tau^{-1} e^{-s}\left(v_0\otimes\nabla\left(\frac{\rho_0}{\rho_\tau}\right)\right)^a+\mathcal{S}_{rot}\right]\cdot(\tau^{-1}e^{-s}J_0+\xi_\tau)dx.
\end{align*}

Now we consider the integral containing the operator $\mathcal{B}$. Recall from \eqref{eq:opB} that we have
\begin{align*}
&\mathcal{B}(\bar\rho,e^{-s}\diver J_0+\rho_{\mathcal{O},1},r)\\
=&\frac14\nabla\triangle r-\frac14\diver\left[\frac{2(\nabla\bar\rho \otimes\nabla r)^s+\nabla \eta_\tau\otimes \nabla \eta_\tau}{\rho_\tau}\right]+\mathcal{B}_2(r,\nabla r).
\end{align*}
where
\begin{equation}\label{eq:B2}
\begin{aligned}
\mathcal{B}_2(r,\nabla r)=&\frac14\diver\left[\frac{r\nabla \bar\rho\otimes\nabla \bar\rho }{\bar\rho^2}\right]-\frac{1}{2\tau^2}\diver\left[\nabla \bar\rho\otimes\nabla \bar\rho\int_0^{\tau \eta_\tau}\frac{\tau \eta_\tau-g}{(\bar\rho+g)^3}dg\right]\\
&+\frac12\diver\left[\frac{(\nabla \bar\rho\otimes\nabla(e^{-s}\diver J_0+\rho_{\mathcal{O},1}))^sr_\tau}{\bar\rho^2}\right]\\
&-\frac{1}{\tau}\diver\left[(\nabla \bar\rho\otimes\nabla(e^{-s}\diver J_0+\rho_{\mathcal{O},1}))^s\int_0^{\tau \eta_\tau}\frac{\tau \eta_\tau-g}{(\bar\rho+g)^3}dg\right]\\
&-\nabla [p'(\bar\rho)r]-\tau^{-2}\nabla\int_0^{\tau \eta_\tau} p''(\bar\rho+g)(\tau \eta_\tau-g)dg\\
&-r\nabla V(\bar\rho)-\bar\rho\nabla V(r)-\eta_\tau\nabla V(\eta_\tau)
\end{aligned}
\end{equation}
and
\[
\eta_\tau=e^{-s}\diver J_0+\rho_{\mathcal{O},1}+\tau r.
\]
We first treat the first two leading order terms in the right hand side of $\mathcal{B}$,
\begin{align*}
\int_{\T^d} \frac{\overline{\mathcal{R}}}{\rho_\tau}\cdot\nabla\triangle rdx-\int_{\T^d} \frac{\overline{\mathcal{R}}}{\rho_\tau}\cdot\diver\left[\frac{2(\nabla\bar\rho \otimes\nabla r)^s+\nabla \eta_\tau\otimes \nabla \eta_\tau}{\rho_\tau}\right]dx.
\end{align*}
By integrating by parts and \eqref{eq:dtr2}, we have
\begin{align*}
\int_{\T^d} \frac{\overline{\mathcal{R}}}{\rho_\tau}\cdot\nabla\triangle rdx=&-\int_{\T^d} \frac{\diver \overline{\mathcal{R}}}{\rho_\tau}\triangle rdx+\int_{\T^d} \frac{\overline{\mathcal{R}}}{\rho_\tau^2}\cdot\nabla\rho_\tau\triangle r dx\\
=&\int_{\T^d} \frac{\nabla\diver \overline{\mathcal{R}}}{\rho_\tau}\cdot\nabla rdx-\int_{\T^d} \frac{\diver \overline{\mathcal{R}}}{\rho_\tau^2}\nabla r\cdot\nabla\rho_\tau dx+\int_{\T^d} \frac{\overline{\mathcal{R}}}{\rho_\tau^2}\cdot\nabla\rho_\tau\triangle r  dx\\
=&-\frac{d}{dt}\int_{\T^d}\frac{|\nabla r|^2}{2\rho_\tau}dx-\int_{\T^d}\frac{|\nabla r|^2}{2\rho_\tau^2}\d_t\rho_\tau dx+\int_{\T^d} \frac{\nabla\diver \mathcal{S}_\rho}{\rho_\tau}\cdot\nabla rdx\\
&-2\int_{\T^d} \frac{\overline{\mathcal{R}}\cdot\nabla\rho_\tau}{\rho_\tau^3}\nabla r\cdot\nabla\rho_\tau dx+\int_{\T^d} \frac{\overline{\mathcal{R}}}{\rho_\tau^2}\cdot(I_d\triangle+\nabla^2)r\cdot\nabla\rho_\tau dx\\
&+\int_{\T^d} \frac{\overline{\mathcal{R}}}{\rho_\tau^2}\cdot\nabla^2\rho_\tau\cdot\nabla r dx,
\end{align*}
where $I_d$ is the $d-$dimensional identity matrix. By the expansion of $\rho_\tau$, we notice that
\begin{align*}
&-\int_{\T^d} \frac{\overline{\mathcal{R}}}{\rho_\tau}\cdot\diver\left[\frac{2(\nabla\bar\rho\otimes \nabla r)^s+\nabla \eta_\tau\otimes \nabla \eta_\tau}{\rho_\tau}\right]dx\\
=&\int_{\T^d} \nabla\rho_\tau\cdot\left[\frac{2(\nabla\bar\rho\otimes \nabla r)^s+\nabla \eta_\tau\otimes \nabla \eta_\tau}{\rho_\tau^2}\right]\cdot\frac{\overline{\mathcal{R}}}{\rho_\tau} dx\\
&-\int_{\T^d} \frac{\overline{\mathcal{R}}}{\rho_\tau}\cdot\frac{\triangle\bar\rho\nabla r+\triangle (e^{-s}\diver J_0+\rho_{\mathcal{O},1}+\tau r)\nabla\eta_\tau}{\rho_\tau}dx\\
&-\int_{\T^d} \frac{\overline{\mathcal{R}}}{\rho_\tau}\cdot\frac{\nabla\bar\rho\cdot\nabla^2 r+\nabla\eta_\tau\cdot\nabla^2(e^{-s}\diver J_0+\rho_{\mathcal{O},1}+\tau r)}{\rho_\tau}dx\\
&-\int_{\T^d} \frac{\overline{\mathcal{R}}}{\rho_\tau}\cdot\frac{\triangle r\nabla\bar\rho+\nabla r\cdot\nabla^2\bar\rho}{\rho_\tau}dx.
\end{align*}
Recall that $\rho_\tau=\bar\rho+\tau\eta_\tau$, then we can re-organize the integrands of the right hand side above as
\begin{align*}
=&\int_{\T^d} \nabla\rho_\tau\cdot\left[\frac{2(\nabla\bar\rho\otimes \nabla r)^s+\nabla \eta_\tau\otimes \nabla \eta_\tau}{\rho_\tau^2}\right]\cdot\frac{\overline{\mathcal{R}}}{\rho_\tau} dx\\
&-\int_{\T^d} \frac{\overline{\mathcal{R}}}{\rho_\tau}\cdot\frac{\nabla r\cdot(I_d\triangle+\nabla^2 )\bar\rho}{\rho_\tau}dx-\int_{\T^d} \frac{\overline{\mathcal{R}}}{\rho_\tau}\cdot\frac{\nabla \eta_\tau\cdot(I_d\triangle+\nabla^2)(e^{-s}\diver J_0+\rho_{\mathcal{O},1})}{\rho_\tau}dx\\
&-\int_{\T^d} \frac{\overline{\mathcal{R}}}{\rho_\tau}\cdot\frac{\nabla\rho_\tau\cdot(I_d\triangle+\nabla^2)r}{\rho_\tau}dx
\end{align*}
Summarizing the identities above and noticing that the integrals containing $(I_d\triangle+\nabla^2)r$ cancel, we obtain
\begin{align*}
&\int_{\T^d} \frac{\overline{\mathcal{R}}}{\rho_\tau}\cdot\nabla\triangle rdx-\int_{\T^d} \frac{\overline{\mathcal{R}}}{\rho_\tau}\cdot\diver\left[\frac{2(\nabla\bar\rho\otimes \nabla r)^s+\nabla \eta_\tau\otimes \nabla \eta_\tau}{\rho_\tau}\right]dx\\
=&-\frac{d}{dt}\int_{\T^d}\frac{|\nabla r|^2}{2\rho_\tau}dx-\int_{\T^d}\frac{|\nabla r|^2}{2\rho_\tau^2}\d_t\rho_\tau dx+\int_{\T^d} \frac{\nabla\diver \mathcal{S}_\rho}{\rho_\tau}\cdot\nabla rdx+\int_{\T^d}\frac{\overline{\mathcal{R}}}{\rho_\tau}\cdot\mathcal{B}_3(\nabla r)dx,
\end{align*}
where
\begin{equation}\label{eq:B3}
\begin{aligned}
\mathcal{B}_3(\nabla r)=&(\nabla\rho_\tau\cdot\nabla r)\nabla \rho_\tau+\nabla^2\rho_\tau\cdot\nabla r+(I_d\triangle+\nabla^2)\bar\rho\cdot\nabla r\\
&-\frac{\nabla \eta_\tau}{\rho_\tau}\cdot(I_d\triangle+\nabla^2)(e^{-s}\diver J_0+\rho_{\mathcal{O},1}),
\end{aligned}
\end{equation}
Summarizing the identities above, we obtain \eqref{eq:dtEr}.
\end{proof}

\begin{prop}
Under the regularity Assumption \ref{asmp:RA}, there exist $\mathcal{S}_j$, $j=1,2,3$, such that identity \eqref{eq:dtEr} can be bounded by
\begin{equation}\label{eq:estE}
\eqref{eq:dtEr}\leq \frac34\int_{\T^d} \frac{\overline{|\mathcal{R}}|^2}{\rho_\tau}dx+\mathcal{S}_1(t)E_\tau(r,\overline{\mathcal{R}})+\tau^2\mathcal{S}_2(t)E_\tau(r,\overline{\mathcal{R}})^\frac32+\mathcal{S}_3(t),
\end{equation}
where $\mathcal{S}_j$, $j=1,2,3$, satisfy
\begin{equation}\label{eq:prop8_1}
\begin{aligned}
\mathcal{S}_1(t)\leq  &(1+\tau^{-2}e^{-\frac{t}{\tau^2}})\\
&\cdot C(\delta,\|\rho_0\|_{H^{3+\frac{d}{2}}_x},\|J_0\|_{H^{3+\frac{d}{2}}},\|\rho_\tau\|_{H^{2+\frac{d}{2}}},\|J_\tau\|_{H^2_x},\|\bar \rho\|_{H^{4+\frac{d}{2}}_x},\|\rho_{\mathcal{O},1}\|_{H^{4+\frac{d}{2}}_x}),
\end{aligned}
\end{equation}
\begin{equation}\label{eq:prop8_2}
\mathcal{S}_2(t)\leq  C(\delta,\|J_0\|_{H^{2+\frac{d}{2}}_x},\|\bar\rho\|_{H^{2+\frac{d}{2}}_x},\|\rho_{\mathcal{O},1}\|_{H^{1+\frac{d}{2}}_x}),
\end{equation}
and
\begin{equation}\label{eq:prop8_3}
\begin{aligned}
\mathcal{S}_3(t)\leq & (1+\tau^{-4}e^{-\frac{4t}{\tau^2}}t+\tau^{-2}e^{-\frac{2t}{\tau^2}})\\
&\cdot C(\delta,\|\rho_0\|_{H^5_x},\|J_0\|_{H^{3+\frac{d}{2}}_x},\|\rho_\tau\|_{H^3_x},\|J_\tau\|_{L^2_tH^2_x},\|\bar\rho\|_{H^7_x},\|\rho_{\mathcal{O},1}\|_{H^7_x}).
\end{aligned}
\end{equation}
\end{prop}

\begin{proof}
In this proof we will repetitively use the following embedding inequalities on $\T^d$,
\[
\|f\|_{L^4_x}\leq C\|f\|_{H^1_x},\quad \|f\|_{L^\infty_x}\leq C\|f\|_{H^{\frac{d}{2}}_x}.
\]
First, by Cauchy-Schwarz inequality we write
\[
\int_{\T^d} \frac{\overline{\mathcal{R}}}{\rho_\tau}\cdot\mathcal{S}_J(t,x)dx\leq \frac{1}{16}\int_{\T^d} \frac{|\overline{\mathcal{R}}|^2}{\rho_\tau}dx+4\int_{\T^d} |\mathcal{S}_J|^2dx,
\]
where the last term should be bounded by $\mathcal{S}_3(t)$. 
By definition \eqref{eq:S2} of $\mathcal{S}_J$, it follows that
\begin{equation}\label{eq:int_A}
\begin{aligned}
\int_{\T^d} |\mathcal{S}_J|^2dx\leq &  \tau^{-2}e^{-2s}\int_{\T^d} |\mathcal{A}(\bar\rho,\diver J_0)|^2 dx\\
&+\tau^{-4}e^{-4s}\int_{\T^d} |\diver[J_0\otimes J_0(\rho_0^{-1}-\rho_\tau^{-1})]|^2 dx\\
&+\int_{\T^d} \left|\diver\left(\frac{2\tau^{-1}e^{-s}\xi_\tau\otimes J_0+\xi_\tau\otimes \xi_\tau}{\rho_\tau}\right)\right|^2 dx\\
&+\int_{\T^d} |\d_tJ_{\mathcal{O},0}|^2dx+\tau^2\int_{\T^d} |\d_tJ_{\mathcal{O},1}|^2dx.
\end{aligned}
\end{equation}
Operator $\mathcal{A}(\cdot,\cdot)$ is given by \eqref{eq:opA}, then we have
\[
\int_{\T^d} |\mathcal{A}(\bar\rho,\diver J_0)|^2 dx\leq C(\delta,\|\bar\rho\|_{H^3_x},\|J_0\|_{H^4_x}).
\]
For the second integral in the right hand side of \eqref{eq:int_A}, noticing that $\rho_0$ is the initial data of $\rho_\tau$, we can write
\[
\rho_0^{-1}-\rho_\tau^{-1}=\int_0^t\frac{\d_s\rho_\tau}{\rho_\tau^2}ds.
\]
Thus by using $\d_s\rho_\tau=-\diver J_\tau$, it gives
\begin{align*}
\int_{\T^d} |\diver[J_0\otimes J_0(\rho_0^{-1}-\rho_\tau^{-1})]|^2 dx\leq & C(\delta,\|J_0\|_{H^{1+\frac{d}{2}}_x})\left(\int_0^t\|\d_s\rho_\tau\|_{H^1_x}ds\right)^2\\
\leq & C(\delta,\|J_0\|_{H^{1+\frac{d}{2}}_x},\|J_\tau\|_{L^2_tH^2_x})t.
\end{align*}
The third term in \eqref{eq:int_A} can be bounded by
\[
\int_{\T^d} \left|\diver\left(\frac{2\tau^{-1}e^{-s}\xi_\tau\otimes J_0+\xi_\tau\otimes \xi_\tau}{\rho_\tau}\right)\right|^2 dx\leq C(\delta,\|J_0\|_{H^2_x},\|\rho_\tau\|_{H^2_x},\|\xi_\tau\|_{H^2_x})(1+\tau^{-2}e^{-2s}).
\]
By using definition \eqref{eq:xi} of $\xi_\tau$, we have
\begin{equation}\label{eq:bd_xi}
\|\xi_\tau\|_{H^2_x}\leq C(\delta,\|\rho_0\|_{H^4_x},\|J_0\|_{H^2_x},\|\bar J\|_{H^2_x},\|J_{\mathcal{O},1}\|_{H^2_x}),
\end{equation}
and further by \eqref{eq:barJ} and \eqref{eq:1equ}, it follows that
\begin{align*}
\|\bar J\|_{H^2_x}\leq C(\delta,\|\bar\rho\|_{H^5_x}),\quad \|J_{\mathcal{O},1}\|_{H^2_x}\leq C(\delta,\|\rho_{\mathcal{O},1}\|_{H^5_x}).
\end{align*}
Last, by using \eqref{eq:0equ} and \eqref{eq:1equ}, we have
\begin{align*}
\int_{\T^d} |\d_tJ_{\mathcal{O},0}|^2dx= &\int_{\T^d} |\d_t\bar J(\bar\rho)|^2dx\\
\leq & C(\delta,\|\bar\rho\|_{H^3_x},\|\d_t\bar\rho\|_{H^3_x})\leq C(\delta,\|\bar\rho\|_{H^7_x}),
\end{align*}
and
\begin{align*}
\int_{\T^d} |\d_tJ_{\mathcal{O},1}|^2dx= &\int_{\T^d} |\d_t\mathcal{A}(\bar\rho,\rho_{\mathcal{O},1})|^2dx\\
\leq & C(\delta,\|\bar\rho\|_{H^3_x},\|\d_t\bar\rho\|_{H^3_x},\|\rho_{\mathcal{O},1}\|_{H^3_x},\|\d_t\rho_{\mathcal{O},1}\|_{H^3_x})\\
\leq & C(\delta,\|\bar\rho\|_{H^6_x},\|\rho_{\mathcal{O},1}\|_{H^7_x}).
\end{align*}
By collecting the estimates above, we obtain
\begin{equation}\label{eq:estS_1}
\begin{aligned}
\int_{\T^d} |\mathcal{S}_J|^2dx\leq & C(\delta,\|\rho_0\|_{H^4_x},\|J_0\|_{H^3_x},\|\rho_\tau\|_{H^2_x},\|J_\tau\|_{L^2_tH^2_x},\|\bar\rho\|_{H^7_x},\|\rho_{\mathcal{O},1}\|_{H^7_x})\\
&\cdot(1+\tau^{-4}e^{-4s}t+\tau^{-2}e^{-2s}).
\end{aligned}
\end{equation}

The second term in the right hand side of \eqref{eq:dtEr} is bounded by
\[
-\frac{\tau^2}{2}\int_{\T^d}\frac{|\overline{\mathcal{R}}|^2}{\rho_\tau^2}\diver(\tau^{-1}e^{-s}J_0+\xi_\tau)dx\leq E_\tau(r,\overline{\mathcal{R}})\|\rho_\tau^{-1}\diver(\tau^{-1}e^{-s}J_0+\xi_\tau)\|_{L^\infty_x},
\]
and this term is controlled by $\mathcal{S}_1(t) E_\tau(r,\overline{\mathcal{R}})$. 
By definition \eqref{eq:bd_xi} of $\xi_\tau$, we have
\begin{equation}\label{eq:estS_2}
\begin{aligned}
\|\rho_\tau^{-1}\diver(\tau^{-1}e^{-s}J_0+\xi_\tau)\|_{L^\infty_x}\leq &C(\delta,\|J_0\|_{H^{1+\frac{d}{2}}_x},\|\xi_\tau\|_{H^{1+\frac{d}{2}}_x})(1+\tau^{-1}e^{-s})\\
\leq & C(\delta,\|\rho_0\|_{H^{3+\frac{d}{2}}_x},\|J_0\|_{H^{1+\frac{d}{2}}_x},\|\bar \rho\|_{H^{4+\frac{d}{2}}_x},\|\rho_{\mathcal{O},1}\|_{H^{4+\frac{d}{2}}_x})\\
&\cdot(1+\tau^{-1}e^{-s}).
\end{aligned}
\end{equation}

Now we turn to the third term in the right hand side of \eqref{eq:dtEr}. We have
\begin{equation}\label{eq:bd_curl}
\begin{aligned}
&\int_{\T^d}\frac{\overline{\mathcal{R}}}{\rho_\tau}\cdot\left[-\tau^{-1} e^{-s}\left(v_0\otimes\nabla\left(\frac{\rho_0}{\rho_\tau}\right)\right)^a+\mathcal{S}_{rot}\right]\cdot(\tau^{-1}e^{-s}J_0+\xi_\tau)dx\\
\leq & \frac{1}{16}\int_{\T^d}\frac{|\overline{\mathcal{R}}|^2}{\rho_\tau}dx+C\tau^{-4}e^{-4s}\int_{\T^d}\frac{|J_0|^4}{\rho_\tau}\left|\nabla\left(\frac{\rho_0}{\rho_\tau}\right)\right|^2dx+C\tau^{-2}e^{-2s}\int_{\T^d}\frac{|\xi_\tau|^2}{\rho_\tau}\left|\nabla\left(\frac{\rho_0}{\rho_\tau}\right)\right|^2dx\\
&+C\tau^{-2}e^{-2s}\int_{\T^d}\frac{|J_0|^2}{\rho_\tau}|\mathcal{S}_{rot}|^2dx+C\int_{\T^d}\frac{|\xi_\tau|^2}{\rho_\tau}|\mathcal{S}_{rot}|^2dx,
\end{aligned}
\end{equation}
and it will bounded by $\mathcal{S}_3(t)$. 
Now we estimate the right hand side of \eqref{eq:bd_curl}. We first notice that
\[
\nabla\left(\frac{\rho_0}{\rho_\tau}\right)=\nabla\left(\frac{\rho_0-\rho_\tau}{\rho_\tau}\right)=\nabla\left(\frac{-\int_0^t\d_s\rho_\tau ds}{\rho_\tau}\right),
\]
therefore we have
\begin{align*}
\int_{\T^d}\frac{|J_0|^4}{\rho_\tau}\left|\nabla\left(\frac{\rho_0}{\rho_\tau}\right)\right|^2dx\leq & C(\delta, \|J_0\|_{H^\frac{d}{2}_x})\left(\int_0^t\|\d_s\rho_\tau\|_{H^1_x}ds\right)^2\\
\leq & C(\delta,\|J_0\|_{H^{\frac{d}{2}}_x},\|J_\tau\|_{L^2_tH^2_x})t.
\end{align*}
For  the third term in the right hand side of \eqref{eq:bd_curl}, we bound it by
\begin{align*}
\int_{\T^d}\frac{|\xi_\tau|^2}{\rho_\tau}\left|\nabla\left(\frac{\rho_0}{\rho_\tau}\right)\right|^2dx\leq & C(\delta,\|\rho_0\|_{H^2_x},\|\rho_\tau\|_{H^2_x},\|\xi_\tau\|_{H^1_x})\\
\leq & C(\delta,\|\rho_0\|_{H^3_x},\|J_0\|_{H^1_x},\|\rho_\tau\|_{H^2_x},\|\bar\rho\|_{H^4_x},\|\rho_{\mathcal{O},1}\|_{H^4_x}).
\end{align*}
Recalling \eqref{eq:S_curl} the definition of $\mathcal{S}_{rot}$, we have
\begin{align*}
\int_{\T^d}\frac{|J_0|^2}{\rho_\tau}|\mathcal{S}_{rot}|^2dx\leq & C(\delta,\|J_0\|_{H^1_x},\|\mathcal{S}_{rot}\|_{H^1_x}),
\end{align*}
and by \eqref{eq:S_curl}, it gives that
\begin{align*}
\|\mathcal{S}_{rot}\|_{H^1_x}\leq & C(\delta,\|J_0\|_{H^3_x},\|\rho_\tau\|_{H^3_x},\|\bar J(\rho_0)\|_{H^2_x},\|\bar J\|_{H^2_x},\|J_{\mathcal{O},1}\|_{H^2_x})\\
\leq & C(\delta,\|\rho_0\|_{H^5_x},\|J_0\|_{H^3_x},\|\rho_\tau\|_{H^3_x},,\|\bar\rho\|_{H^5_x},\|\rho_{\mathcal{O},1}\|_{H^5_x}).
\end{align*}
The last term of \eqref{eq:bd_curl} can be bounded by
\begin{align*}
\int_{\T^d}\frac{|\xi_\tau|^2}{\rho_\tau}|\mathcal{S}_{rot}|^2dx\leq &C(\delta,\|\xi\|_{H^1_x},\|\mathcal{S}_{rot}\|_{H^1_x})\\
\leq & C(\delta,\|\rho_0\|_{H^5_x},\|J_0\|_{H^3_x},\|\rho_\tau\|_{H^3_x},\|\bar\rho\|_{H^5_x},\|\rho_{\mathcal{O},1}\|_{H^5_x}).
\end{align*}
Thus, we obtain 
\begin{equation}\label{eq:estS_3}
\begin{aligned}
\eqref{eq:bd_curl}\leq & C(\delta,\|\rho_0\|_{H^5_x},\|J_0\|_{H^3_x},\|\rho_\tau\|_{H^3_x},\|J_\tau\|_{L^2_tH^2_x},\|\bar\rho\|_{H^5_x},\|\rho_{\mathcal{O},1}\|_{H^5_x})\\
&\cdot (1+\tau^{-4}e^{-4s}t+\tau^{-2}e^{-2s}).
\end{aligned}
\end{equation}

For the third line of the right hand side of \eqref{eq:dtEr}, we first have
\begin{equation}\label{eq:estS_4}
-\int_{\T^d}\frac{|\nabla r|^2}{2\rho_\tau^2}\d_t\rho_\tau dx\leq C(\delta,\|\d_t \rho_\tau\|_{H^1_x})E_\tau(r,\overline{\mathcal{R}})\leq C(\delta,\|J_\tau\|_{H^2_x})E_\tau(r,\overline{\mathcal{R}}).
\end{equation}
By definition \eqref{eq:S_rho}, we also have
\begin{equation}\label{eq:estS_5}
\begin{aligned}
\int_{\T^d} \frac{\nabla\diver \mathcal{S}_\rho}{\rho_\tau}\cdot\nabla rdx\leq & \int_{\T^d} \frac{|\nabla\diver \mathcal{S}_\rho|^2}{\rho_\tau}dx\int_{\T^d}\frac{|\nabla r|^2}{\rho_\tau}dx\\
\leq & (1+\tau^{-2}e^{-s})C(\delta,\|\bar J(\rho_0)\|_{H^2_x},\|J_0\|_{H^4_x})E_\tau(r,\overline{\mathcal{R}})\\
\leq & (1+\tau^{-2}e^{-s})C(\delta,\|\bar \rho_0\|_{H^5_x},\|J_0\|_{H^4_x})E_\tau(r,\overline{\mathcal{R}}).
\end{aligned}
\end{equation}

Now we collect the estimates related to the operator $\mathcal{B}_2$ and $\mathcal{B}_3$. We first give the control of
\begin{align*}
\int_{\T^d} \frac{\overline{\mathcal{R}}}{\rho_\tau}\cdot\mathcal{B}_2(r,\nabla r)dx
\end{align*}
where $\mathcal{B}_2(r,\nabla r)$ is defined by \eqref{eq:B2}. There are two type of terms in ${B}_2(r,\nabla r)$, which contribute to the source terms $\mathcal{S}_1(t)$ and $\mathcal{S}_2(t)$ respectively. First, we have
\begin{equation}\label{eq:ieB2_1}
\begin{aligned}
\int_{\T^d} \frac{\overline{\mathcal{R}}}{\rho_\tau}\cdot\diver\left[\frac{r\nabla \bar\rho\otimes\nabla \bar\rho }{\bar\rho^2}\right]dx=&\int_{\T^d} \frac{\overline{\mathcal{R}}}{\rho_\tau}\cdot\nabla r\left(\frac{\nabla \bar\rho\otimes\nabla \bar\rho }{\bar\rho^2}\right)dx\\
&+\int_{\T^d} \frac{\overline{\mathcal{R}}r}{\rho_\tau}\cdot\left[\frac{(\triangle I_d+\nabla^2) \bar\rho }{\bar\rho^2}-2\frac{|\nabla\bar\rho|^2}{\bar\rho^3}I_d\right]\cdot\nabla\bar\rho dx.
\end{aligned}
\end{equation}
By the definition of $r$, it is straightforward to check $\int_{\T^d} r\,dx=0$, then by Poincar\'e inequality we have
\[
\|r\|_{L^2_x}\leq C\|\nabla r\|_{L^2_x}.
\]
Thus by using property (c) of $\bar\rho$, we can control the right hand side of \eqref{eq:ieB2_1} by
\begin{align*}
\leq & C\int_{\T^d}\rho_\tau^{-1}(\overline{\mathcal{R}}^2+|\nabla r|^2)dx\left\{\|\frac{\nabla \bar\rho\otimes\nabla \bar\rho }{\bar\rho^2}\|_{L^\infty_x}\right.\\
&\hspace{1cm}\left.+\|\left[\frac{(\triangle I_d+\nabla^2) \bar\rho }{\bar\rho^2}-2\frac{|\nabla\bar\rho|^2}{\bar\rho^3}I_d\right]\cdot\nabla\bar\rho\|_{L^\infty_x}\right\}\\
\leq &  C(\delta,\|\bar\rho\|_{H^{2+\frac{d}{2}}}) E(r,\overline{\mathcal{R}}).
\end{align*}
Similarly we have
\begin{align*}
\int_{\T^d} \frac{\overline{\mathcal{R}}}{\rho_\tau}\cdot\diver\left[\frac{(\nabla \bar\rho\otimes\nabla(e^{-s}\diver J_0+\rho_{\mathcal{O},1}))^sr_\tau}{\bar\rho^2}\right]dx\\
\leq C(\delta,\|\bar\rho\|_{H^{2+\frac{d}{2}}},\|J_0\|_{H^{2+\frac{d}{2}}},\|\rho_{\mathcal{O},1}\|_{H^{1+\frac{d}{2}}})E(r,\overline{\mathcal{R}})
\end{align*}
and
\begin{align*}
&\int_{\T^d} \frac{\overline{\mathcal{R}}}{\rho_\tau}\cdot\left[-\nabla(p'(\bar\rho)r)-r\nabla V(\bar\rho)-\bar\rho\nabla V(r)\right]\leq C(\|\bar\rho\|_{H^{2+\frac{d}{2}}})E(r,\overline{\mathcal{R}}).
\end{align*}
Now we treat another type of terms in $\mathcal{B}_2(r,\nabla r)$, which contain quadratic terms of $(r,\nabla r)$. We take the integral
\[
-\frac{1}{2\tau^2}\int_{\T^d} \frac{\overline{\mathcal{R}}}{\rho_\tau}\cdot\diver\left[\nabla \bar\rho\otimes\nabla \bar\rho\int_0^{\tau \eta_\tau}\frac{\tau \eta_\tau-g}{(\bar\rho+g)^3}dg\right]dx
\]
as an example, and the other parts are controlled in a similar way. Recall that
\[
\eta_\tau=e^{-s}\diver J_0+\rho_{\mathcal{O},1}(t)+\tau r,
\]
then we have
\[
\int_0^{\tau \eta_\tau}\frac{\tau \eta_\tau-g}{(\bar\rho+g)^3}dg\leq C(\delta)\tau^2\eta_\tau^2\leq C(\delta)\tau^2[e^{-2s}(\diver J_0)^2+\rho_{\mathcal{O},1}^2+\tau^2 r^2]
\]
and
\begin{align*}
\nabla\left(\int_0^{\tau \eta_\tau}\frac{\tau \eta_\tau-g}{(\bar\rho+g)^3}dg\right)\leq & C(\delta)\tau^2\left(|\eta_\tau\nabla\eta_\tau|+\eta_\tau^2|\nabla\bar\rho|\right)\\
\leq & C(\delta)\tau^2(e^{-2s}|\diver J_0||\nabla\diver J_0|+|\rho_{\mathcal{O},1}||\nabla \rho_{\mathcal{O},1}|+\tau^2|r\nabla r|)\\
&+C(\delta)\tau^2|\nabla\bar\rho|[e^{-2s}(\diver J_0)^2+\rho_{\mathcal{O},1}^2+\tau^2 r^2].
\end{align*}
Thus we obtain
\begin{align*}
&-\frac{1}{2\tau^2}\int_{\T^d} \frac{\overline{\mathcal{R}}}{\rho_\tau}\cdot\diver\left[\nabla \bar\rho\otimes\nabla \bar\rho\int_0^{\tau \eta_\tau}\frac{\tau \eta_\tau-g}{(\bar\rho+g)^3}dg\right]dx\\
\leq & \frac{1}{16}\int_{\T^d} \frac{\overline{|\mathcal{R}}|^2}{\rho_\tau}dx+C(\delta,\|J_0\|_{H^{2+\frac{d}{2}}_x},\|\bar\rho\|_{H^{2+\frac{d}{2}}_x},\|\rho_{\mathcal{O},1}\|_{H^{1+\frac{d}{2}}_x})\\
&+C(\delta,\|\bar\rho\|_{H^{2+\frac{d}{2}}_x})\tau^2E(r,\overline{\mathcal{R}})^\frac32.
\end{align*}
For the other quadratic terms, we also have
\begin{align*}
&-\frac{1}{\tau}\int_{\T^d} \frac{\overline{\mathcal{R}}}{\rho_\tau}\cdot\diver\left[(\nabla \bar\rho\otimes\nabla(e^{-s}\diver J_0+\rho_{\mathcal{O},1}))^s\int_0^{\tau \eta_\tau}\frac{\tau \eta_\tau-g}{(\bar\rho+g)^3}dg\right]dx\\
\leq &\frac{\tau}{16}\int_{\T^d} \frac{\overline{|\mathcal{R}}|^2}{\rho_\tau}dx+C(\delta,\|J_0\|_{H^{2+\frac{d}{2}}_x},\|\bar\rho\|_{H^{2+\frac{d}{2}}_x},\|\rho_{\mathcal{O},1}\|_{H^{1+\frac{d}{2}}_x})\tau\\
&+C(\delta,\|J_0\|_{H^{2+\frac{d}{2}}_x},\|\bar\rho\|_{H^{2+\frac{d}{2}}_x},\|\rho_{\mathcal{O},1}\|_{H^{1+\frac{d}{2}}_x})\tau^3E(r,\overline{\mathcal{R}})^\frac32,
\end{align*}
\begin{align*}
&-\tau^{-2}\int_{\T^d} \frac{\overline{\mathcal{R}}}{\rho_\tau}\cdot\nabla\left[\int_0^{\tau \eta_\tau} p''(\bar\rho+g)(\tau \eta_\tau-g)dg\right]dx\\
\leq & \frac{1}{16}\int_{\T^d} \frac{\overline{|\mathcal{R}}|^2}{\rho_\tau}dx+C(\|J_0\|_{H^{2+\frac{d}{2}}_x},\|\bar\rho\|_{H^{2+\frac{d}{2}}_x},\|\rho_{\mathcal{O},1}\|_{H^{1+\frac{d}{2}}_x})\\
&+C(\delta,\|\bar\rho\|_{H^{2+\frac{d}{2}}_x})\tau^2E(r,\overline{\mathcal{R}})^\frac32.
\end{align*}
and
\begin{align*}
-\int_{\T^d} \frac{\overline{\mathcal{R}}}{\rho_\tau}\cdot\eta_\tau\nabla V(\eta_\tau)dx\leq &\frac{1}{16}\int_{\T^d} \frac{\overline{|\mathcal{R}}|^2}{\rho_\tau}dx+C(\|J_0\|_{H^{2+\frac{d}{2}}_x},\|\bar\rho\|_{H^{2+\frac{d}{2}}_x},\|\rho_{\mathcal{O},1}\|_{H^{1+\frac{d}{2}}_x})\\
&+C(\delta,\|\bar\rho\|_{H^{2+\frac{d}{2}}_x})\tau^2E(r,\overline{\mathcal{R}})^\frac32.
\end{align*}
Thus we obtain
\begin{equation}\label{eq:estS_6}
\begin{aligned}
\int_{\T^d} \frac{\overline{\mathcal{R}}}{\rho_\tau}\cdot\mathcal{B}_2(r,\nabla r)dx\leq & \frac14 \int_{\T^d} \frac{\overline{|\mathcal{R}}|^2}{\rho_\tau}dx+C(\delta,\|\bar\rho\|_{H^{2+\frac{d}{2}}},\|J_0\|_{H^{2+\frac{d}{2}}},\|\rho_{\mathcal{O},1}\|_{H^{1+\frac{d}{2}}})E(r,\overline{\mathcal{R}})\\
&+C(\delta,\|J_0\|_{H^{2+\frac{d}{2}}_x},\|\bar\rho\|_{H^{2+\frac{d}{2}}_x},\|\rho_{\mathcal{O},1}\|_{H^{1+\frac{d}{2}}_x})\tau^2 E(r,\overline{\mathcal{R}})^\frac32\\
&+C(\|J_0\|_{H^{2+\frac{d}{2}}_x},\|\bar\rho\|_{H^{2+\frac{d}{2}}_x},\|\rho_{\mathcal{O},1}\|_{H^{1+\frac{d}{2}}_x}).
\end{aligned}
\end{equation}
Last, we need to provide the bound of 
\[
\int_{\T^d}\frac{\overline{\mathcal{R}}}{\rho_\tau}\cdot\mathcal{B}_3(\nabla r)dx.
\]
By definition \eqref{eq:B3} of $\mathcal{B}_3(\nabla r)$, we obtain
\begin{equation}\label{eq:estS_6}
\begin{aligned}
\int_{\T^d}\frac{\overline{\mathcal{R}}}{\rho_\tau}\cdot\mathcal{B}_3(\nabla r)dx\leq C(\delta,\|J_0\|_{H^{3+\frac{d}{2}}},\|\rho_\tau\|_{H^{2+\frac{d}{2}}},\|\bar\rho\|_{H^{2+\frac{d}{2}}},\|\rho_{\mathcal{O},1}\|_{H^{2+\frac{d}{2}}})E_\tau(r,\overline{\mathcal{R}}).
\end{aligned}
\end{equation}
By summarizing estimates \eqref{eq:estS_1} to \eqref{eq:estS_6}, we obtain \eqref{eq:estE}, and the functions $\mathcal{S}_j$, $j=1,2,3$, can be checked to satisfy the bound \eqref{eq:prop8_1} to \eqref{eq:prop8_3}.
\end{proof}

Now we can finish the estimate of $(r,\overline{\mathcal{R}})$.

\begin{prop}\label{prop:bd_en}
Under Assumption \ref{asmp:RA}, there exists $\tau^*=\tau^*(T_0,\delta,M)$, such that for $0<\tau\leq \tau^*$ the energy functional $E_\tau(r,\overline{\mathcal{R}})$ is bounded by
\begin{equation}\label{eq:propEr}
E_\tau(r,\overline{\mathcal{R}})(t)+\int_0^t\int_{{\T^d}} \frac{|\overline{\mathcal{R}}|^2}{\rho_\tau}dxds\leq e^{M(1+t)}\left[1+E_\tau(0,\overline{\mathcal{R}}(0))\right]
\end{equation}
for $t\in [0,T_0]$.
\end{prop}

\begin{proof}
Since $M$ given in Assumption \ref{asmp:RA} dominates the constants in \eqref{eq:prop8_1} to \eqref{eq:prop8_3}, then by \eqref{eq:estE} we have
\begin{align*}
\frac{d}{dt}E_\tau(r,\overline{\mathcal{R}})+\int_{{\T^d}} \frac{|\overline{\mathcal{R}}|^2}{\rho_\tau}dx\leq M[&(1+\tau^{-2}e^{-\frac{t}{\tau^2}})E_\tau(r,\overline{\mathcal{R}})+\tau^2E_\tau(r,\overline{\mathcal{R}})^\frac32\\
&+1+\tau^{-4}e^{-\frac{4t}{\tau^2}}t+\tau^{-2}e^{-\frac{2t}{\tau^2}}].
\end{align*}
We fist let $\tau^*<E_\tau(0,\overline{\mathcal{R}}_0)^{-\frac12}$. Then by our assumption and the continuity of $E_\tau(r,\overline{\mathcal{R}})$, for $\tau\leq \tau^*$ there exist a $T^*>0$ such that on $[0,T^*]$ we have
\[
\tau^2E_\tau(r,\overline{\mathcal{R}})^\frac12\leq 1.
\]
Then on $[0,T^*]$ it follows that
\[
\frac{d}{dt}E_\tau(r,\overline{\mathcal{R}})+\int_{{\T^d}} \frac{|\overline{\mathcal{R}}|^2}{\rho_\tau}dx\leq h_1(t)E_\tau(r,\overline{\mathcal{R}})+h_2(t),
\]
where
\[
h_1(t)=C(1+\tau^{-2}e^{-\frac{t}{\tau^2}}),\quad h_2(t)=1+\tau^{-4}e^{-\frac{4t}{\tau^2}}t+\tau^{-2}e^{-\frac{2t}{\tau^2}}.
\]
By the Gronwall argument we obtain
\begin{align*}
E_\tau(r,\overline{\mathcal{R}})+\int_0^t\int_{{\T^d}} \frac{|\overline{\mathcal{R}}|^2}{\rho_\tau}dxds\leq e^{\int_0^t h_1(r)dr}E_\tau(0,\overline{\mathcal{R}}(0))+\int_0^t e^{\int_r^t h_1(r_1)dr_1}h_2(r)dr
\end{align*}
where the definitions of $h_j(t)$ give
\[
\int_r^t h_1(r_1)dr_1=M\left(t-r+e^{-\frac{r}{\tau^2}}-e^{-\frac{t}{\tau^2}}\right),
\]
and
\begin{align*}
\int_0^t e^{\int_r^t h_1(r_1)dr_1}h_2(r)dr\leq & \int_0^t e^{M(t-r)}\left(1+\tau^{-4}e^{-\frac{4r}{\tau^2}}r+\tau^{-2}e^{-\frac{2r}{\tau^2}}\right)dr\\
\leq & (1+e^{Mt}).
\end{align*}
Thus we prove \eqref{eq:propEr} on $[0,T^*]$, and the upper bound of $T^*$ is given by the requirement
\[
\tau^2E_\tau(r,\overline{\mathcal{R}})(T^*)^\frac12\leq \tau^2e^{M(1+T^*)}\left[1+E_\tau(0,\overline{\mathcal{R}}(0))\right]^\frac12\leq 1.
\]
Therefore we can update $\tau^*=\tau^*(T_0,\delta,M)$ small such that for $\tau\leq \tau^*$ we have $T^*\geq T_0$.
\end{proof}

Now we are at the point to finish the proof of Theorem \ref{thm:main}.

\begin{proof}[Proof of Theorem \ref{thm:main}]
Following Proposition \ref{prop:bd_en} and Definition \eqref{eq:en_r} of the remainder energy, for $\tau\leq\tau^*$ we have
\[
\|\nabla r\|_{L^\infty_tL^2_x}+\|\tau\overline{\mathcal{R}}\|_{L^\infty_tL^2_x}+\|\overline{\mathcal{R}}\|_{L^2_{t,x}}\leq C(T_0,\delta,M).
\]
First, by \eqref{eq:r_1} and expression \eqref{eq:rho_aprx}, it is straightforward to check 
\[
\int_{\T^d} r(s,t,x)\,dx=0,\quad\textrm{for all }(s,t),
\]
then by Poincar\'e inequality we have
\[
\|r\|_{L^\infty_tH^1_x}\leq C\|\nabla r\|_{L^\infty_tL^2_x}\leq C(T_0,\delta,M).
\]
On the other hand, by Definition \eqref{eq:alter_R} of the modified remainder $\overline{\mathcal{R}}$, we have
\begin{align*}
\overline{\mathcal{R}}-\mathcal{R}=&\tau^{-1}(J_{\mathcal{I},1}-P_{J,1})(s)\\
=&+2\tau^{-1} e^{-s}(1-e^{-s}-s)\diver\left[\frac{J_0\otimes \bar J(\rho_0)}{\rho_0}\right]^s\\
&+2\tau^{-1} e^{-s}(1+e^{-2s}-2e^{-s})\diver\left\{\left[\frac{J_0}{\rho_0}\otimes\diver\left(\frac{J_0\otimes J_0}{\rho_0}\right)\right]^s-\frac{J_0\otimes J_0\diver J_0}{2\rho_0^2}\right\}\\
&-\tau^{-1} e^{-s}(s+1)\mathcal{A}(\rho_0,\diver J_0),
\end{align*}
where the fast time $s=\tau^{-2}t$. Provided the regularity Assumption \ref{asmp:RA}, it follows that
\[
\|\tau(\overline{\mathcal{R}}-\mathcal{R})\|_{L^\infty_tL^2_x}=\|J_{\mathcal{I},1}-P_{J,1}\|_{L^\infty_tL^2_x}\leq C(T_0,\delta,M),
\]
and
\[
\|\overline{\mathcal{R}}-\mathcal{R}\|_{L^2_{t,x}}=\|\tau^{-1}(J_{\mathcal{I},1}-P_{J,1})\|_{L^2_{t,x}}\leq C(T_0,\delta,M).
\]
Therefore $\overline{\mathcal{R}}$ and $\mathcal{R}$ share the same bounds, and we obtain \eqref{eq:main_r}.

For bound \eqref{eq:main_J}, by the asymptotic expansion \eqref{eq:aprx} we have
\[
J_\tau-\tau^{-1}e^{-\frac{t}{\tau^2}}J_0-\bar J=\tau J_{\mathcal{O},1}+\tau^2\overline{\mathcal{R}}.
\]
Combining it with Assumption \eqref{asmp:RA} and the bounds of $\overline{\mathcal{R}}$, it implies
\begin{align*}
\|J_\tau-&\tau^{-1}e^{-\frac{t}{\tau^2}}J_0-\bar J\|_{L^\infty_tL^2_x}\\
&\leq  \tau \|\mathcal{A}(\bar\rho,\rho_{\mathcal{O},1})\|_{L^\infty_tL^2_x}+\tau \|\tau\overline{\mathcal{R}}\|_{L^\infty_tL^2_x}
\leq  C(T_0,\delta,M).
\end{align*}
Also, we have
\[
\rho_\tau-\bar\rho-\tau [e^{-s}\diver J_0+\rho_{\mathcal{O},1}]=\tau^2 r,
\]
which gives \eqref{eq:main_r} by applying the bound of $r$.
\end{proof}
\section*{Acknowledgments}
The first and second authors are members of the GNAMPA group of INdAM. The first author acknowledges financial support from the Italian Ministry of University and Research (MUR) through the Project 2022YXWSLR \emph{Boundary analysis for dispersive and viscous fluids}, and the Excellence Department Project awarded to GSSI, CUP D13C22003740001.

\end{document}